\documentclass[a4paper,10pt]{article}

\usepackage{amsmath,amssymb,amsthm,amscd}
\usepackage[mathscr]{eucal}
\usepackage[all]{xy}
\usepackage{graphicx}
\usepackage{longtable}

\makeatletter
    
    \@addtoreset{equation}{section}
  \makeatother

\newcommand{\mcal}{\mathcal}
\newcommand{\mbf}{\mathbf}

\newcommand{\mfrak}{\mathfrak}
\newcommand{\mbb}{\mathbb}
\newcommand{\mrm}{\mathrm}
\newcommand{\vphi}{\varphi}

\newcommand{\cO}{\mathcal{O}}

\newtheorem{theorem}{Theorem}[section]
\newtheorem{corollary}[theorem]{Corollary}
\newtheorem{lemma}[theorem]{Lemma}
\newtheorem{proposition}[theorem]{Proposition}

\theoremstyle{definition}

\newtheorem{remark}[theorem]{Remark}

\newtheorem*{acknowledgments}{Acknowledgments}

\title{Explicit bounds on torsion of CM abelian varieties over $p$-adic fields 
with values in Lubin-Tate extensions} 

\author{Yoshiyasu Ozeki\footnote{
Faculty of Science, Kanagawa University,
3-27-1 Rokkakubashi, Kanagawa-ku, Yokohama-shi, Kanagawa 221-8686, JAPAN
\endgraf
e-mail: {\tt ozeki@kanagawa-u.ac.jp}}
}

\begin{document}
\maketitle

\begin{abstract}
Let $K$ and $k$ be $p$-adic fields. 
Let $L$ be the composite field of 
$K$ and  a certain Lubin-Tate extension over  $k$
(including the case where $L=K(\mu_{p^{\infty}})$).
In this paper, 
we show that there exists an explicitly described constant $C$,
depending only on $K,k$ and an integer $g \ge 1$,   
which satisfies the following property:
If $A_{/K}$ is a $g$-dimensional CM abelian variety,  
then the order of the $p$-torsion subgroup of $A(L)$
 is bounded by $C$. 
We also give a similar bound in the case where 
$L=K(\sqrt[p^{\infty}]{K})$.
Applying our results, we study  bounds of  orders of  
torsion subgroups of some CM abelian varieties 
over number fields with values in full cyclotomic fields. 
\end{abstract}

%\setcounter{tocdepth}{1}
%     \tableofcontents

%%%%%%%%%%%%%%%%%%%%%%%%%%%%%%%%%%%%%%%%%%%%%%%%%%%%%%%%%%%%%%%%%%%%%%%%%%%%%%%%%%%%%%%%%%%%%%%%%%%%%%%%%%%
%%%%%%%%%%%%%%%%%%%%%%%%%%%%%%%%%%%%%%%%%%%%%%%%%%%%%%%%%%%%%%%%%%%%%%%%%%%%%%%%%%%%%%%%%%%%%%%%%%%%%%%%%%%
%                           1                              %%%%%%%%%%%%%%%%%%%%%%%%%%%%%%%%%%%%%%%%%%%%%%%%
%%%%%%%%%%%%%%%%%%%%%%%%%%%%%%%%%%%%%%%%%%%%%%%%%%%%%%%%%%%%%%%%%%%%%%%%%%%%%%%%%%%%%%%%%%%%%%%%%%%%%%%%%%%
%%%%%%%%%%%%%%%%%%%%%%%%%%%%%%%%%%%%%%%%%%%%%%%%%%%%%%%%%%%%%%%%%%%%%%%%%%%%%%%%%%%%%%%%%%%%%%%%%%%%%%%%%%%

\section{Introduction}

Let $p$ be a  prime number and $K$ a $p$-adic field 
(= a finite extension of $\mbb{Q}_p$). 
It is a theorem of Mattuck \cite{Mat} that, for a $g$-dimensional abelian variety $A$ 
over $K$ and a finite extension $L/K$,  the Mordell-Weil group 
$A(L)$ is isomorphic to 
the direct sum of $\mbb{Z}_p^{\oplus g\cdot [L:\mbb{Q}_p]}$ 
and a finite group.  
Our interest is to study various information about
the torsion subgroup  $A(L)_{\mrm{tor}}$ of $A(L)$. 
For this, Clark and Xarles \cite{CX} gave an explicit upper bound of 
the order of $A(L)_{\mrm{tor}}$ of $A(L)$
in terms of $p, g$ and some numerical invariants of  $L$
if $A$ has anisotropic reduction.
This includes the case where $A$ has potential good reduction
and in this case the existence of a bound can be found in some literatures
(cf.\ \cite{Si2}, \cite{Si3}).
We consider the case where $L/K$ is {\it of infinite degree}.
There are some situations in which the torsion part $A(L)_{\mrm{tor}}$ is finite.
Suppose that $A$ has potential good reduction.
It is a theorem of Imai  \cite{Im} that $A(K(\mu_{p^{\infty}}))_{\mrm{tor}}$  is finite.
Here,  $K(\mu_{p^{\infty}})$ is the extension field of 
$K$ obtained by adjoining   all $p$-power roots of unity.
Moreover,  Kubo and Taguchi showed in \cite{KT} that 
$A(K(\sqrt[p^{\infty}]{K}))_{\mrm{tor}}$  is also finite 
where $K(\sqrt[p^{\infty}]{K})$ is the extension field of $K$
obtained by adjoining   all $p$-power roots of all elements of $K$.
The author showed in \cite{Oz1} 
that there exists a "uniform" and "theoretical" bound of the order of 
$A(K(\sqrt[p^{\infty}]{K}))_{\mrm{tor}}$
under the assumption that $A$ has complex multiplication. 
(Here we say that $A$ has complex multiplication if there exists 
a ring homomorphism $F\to \mbb{Q}\otimes_{\mbb{Z}}\mrm{End}_{\overline{K}}A$
for some algebraic number field $F$ of degree $2g$.)

The main purpose of this paper is to give explicit upper  bounds of 
the orders of  $A(K(\mu_{p^{\infty}}))_{\mrm{tor}}$
and $A(K(\sqrt[p^{\infty}]{K}))_{\mrm{tor}}$ 
for  abelian varieties $A/K$ with
complex multiplication.
For this, we should note that to give an upper bound of the order of the prime-to-$p$ part of $A(K(\mu_{p^{\infty}}))_{\mrm{tor}}$
is not so difficult.  In fact, the reduction map gives an injection 
from the prime-to-$p$ part of the group which we want to study
into certain rational points of the reduction $\bar{A}$ of $A$ 
(if $A$ has good reduction), 
and the order of the target is bounded by the Weil bound.
Hence the essential obstruction for our purpose appears in a study of 
the $p$-part $A(K(\mu_{p^{\infty}}))[p^{\infty}]$ of 
$A(K(\mu_{p^{\infty}}))_{\mrm{tor}}$.

Let us state our main results. For a $p$-adic field $k$ and a uniformizer $\pi$
of $k$, we denote by $k_{\pi}/k$ the Lubin-Tate extension associated with $\pi$
(that is,
$k_{\pi}$ is the extension field of $k$ obtained by adjoining 
all $\pi$-power torsion points of the Lubin-Tate formal group associated with $\pi$).
For example,  we have $k_{\pi}=\mbb{Q}_p(\mu_{p^{\infty}})$ if 
 $k=\mbb{Q}_p$ and $\pi=p$. 
We set $d_L:=[L:\mbb{Q}_p]$ for any $p$-adic field $L$.
For any integer $n>0$, 
we set
\begin{align*}
& \Phi(n):=
\mrm{Max}\{m\in \mbb{Z} \mid \mbox{$\varphi(m)$ divides $2n$} \}, \\
& H(n):=
\mrm{gcd}\{\sharp \mrm{GSp}_{2n}(\mbb{Z}/N\mbb{Z}) \mid 
N\ge 3 \}. 
\end{align*}
Here, $\varphi$ is the Euler's totient function. 
There are some upper bounds related with $H(n)$ and $\Phi(n)$
(see Section \ref{Phi-H}).
It is a theorem of Silverberg \cite{Si1} that we have $H(n)<2(9n)^{2n}$ for any $n>0$. 
It follows from elementary arguments that we have $\Phi(n)<6n\sqrt[3]{n}$ for $n>1$.
Furthermore, a lower bound \eqref{phi:lower} of $\vphi$
proved by Rosser and Schoenfeld  \cite{RS}
gives $\Phi(n)< 4n\log \log n$ for $n>3^{3^9}$.
\if0
The upper bounds of functions $\Phi(n)$ and $H(n)$
are important for our results.
A lot of bounds on $\Phi(n)$ are well-known.
The function $H(n)$ is well-studied in \cite{Si1}
We set 
\begin{align*}
m:=
\left\{
\begin{array}{cl}
3 &\quad 
\mrm{if}\ p\not=3, \cr
4 &\quad \mrm{if}\ p=3,
\end{array}
\right.
\quad 
d_g:=\sharp GL_{2g}(\mbb{Z}/m\mbb{Z})
=
\left\{
\begin{array}{cl}
\prod^{2g-1}_{i=0}(3^{2g}-3^i) &\quad 
\mrm{if}\ p\not=3, \cr
2^{4g^2}\prod^{2g-1}_{i=0}(2^{2g}-2^i) 
&\quad \mrm{if}\ p=3.
\end{array}
\right.
\end{align*}
Note that we have $d_g<m^{4g^2}$.
\fi

\begin{theorem}[= a special case of Theorem \ref{MT:CM:refined}]
\label{MT:CM}
Let $g>0$ be a positive integer. 
Let  $k$ be a $p$-adic field  with residue cardinality $q_k$ and  $\pi$ a uniformizer of $k$.
Assume the following conditions.
\begin{itemize}
\item[{\rm (i)}] $q_k^{-1}\mrm{Nr}_{k/\mbb{Q}_p}(\pi)$ is a root of 
unity\footnote{This condition is equivalent to say that some finite extension of $k_{\pi}$
contains $\mbb{Q}_p(\mu_{p^{\infty}})$ (cf.\ \cite[Lemma 2.7 (2)]{Oz}).};
we denote by $0<\mu<p$ the minimum integer so that 
$(q_k^{-1}\mrm{Nr}_{k/\mbb{Q}_p}(\pi))^{\mu}=1$, and 
\item[{\rm (ii)}] $d_k$ is prime to $(2g)!$.
\end{itemize}
Then, for any $g$-dimensional abelian variety $A$ over 
a $p$-adic field $K$ with complex multiplication, 
we have 
$$
A(Kk_{\pi})[p^{\infty}]\subset A[p^C]
$$
where 
$$
C := 2g^2\cdot (2g)!\cdot \Phi(g)H(g)\cdot \mu \cdot d_{Kk}  
+12g^2-18g+10.
$$
\if0
\begin{align*}
C& := 2g^2\cdot (2g)!\cdot \Phi(g)H(g)\cdot \mu \cdot d_{Kk}  
+12g^2-18g+10 \\
& < 2g^2\cdot (2g)! \cdot m^{4g^2}d_{Kk} (p-1).
\end{align*}
\fi
In particular, we have 
$$
\sharp A(Kk_{\pi})[p^{\infty}]\le p^{2gC}.
$$
\end{theorem}
As an immediate consequence of the theorem above, 
we obtain a result for cyclotomic extensions;
see Corollary \ref{MT:CM:cycl}.
Furthermore, the method of our proof of Theorem \ref{MT:CM}
can be applied to the filed $K(\sqrt[p^{\infty}]{K})$ 
discussed in Kubo and Taguchi, 
which gives a refinement of the main theorem of \cite{Oz1}.
\begin{theorem}
\label{MT:CM:KT}
Let $g>0$ be a positive integer.
For any $g$-dimensional abelian variety $A$ over a $p$-adic field 
$K$ with complex multiplication, 
we have 
$$
A(K(\sqrt[p^{\infty}]{K}))[p^{\infty}] \subset A[p^C]
$$
where 
$$
C := 2g^2 \cdot (2g)!\cdot p^{1+v_p(2)}\cdot 
(\Phi(g)H(g))^2 \cdot p^{v_p(d_K)}d_K
+12g^2-18g+10.
$$
$($Here, $v_p$ is the $p$-adic valuation normalized by $v_p(p)=1$.$)$ 
In particular, we have 
$$
\sharp A(K(\sqrt[p^{\infty}]{K}))[p^{\infty}]\le p^{2gC}.
$$
\end{theorem}
We can consider some further topics. 
For example, we do not know what will happen if we remove the CM assumption from above theorems.
Our proofs in this paper deeply depend on the theory of locally algebraic representations,
which can be  adapted only for abelian representations.
This is the main reason why we can not remove the CM assumption
form our arguments. 
To overcome this obstruction, 
it seems to be helpful for us to study the case of  
(not necessary CM) elliptic curves. We will study this case as a future work.
We are also interested in giving 
the list of the groups that appears as $A(Kk_{\pi})[p^{\infty}]$ or 
$A(K(\sqrt[p^{\infty}]{K}))[p^{\infty}]$. 
However, this should be quite difficult; 
the author does not know such classification results even for $A(K)[p^{\infty}]$.

Combining the cyclotomic case of Theorem \ref{MT:CM}  and 
Ribet's arguments in \cite{KL},
we can obtain a result 
on a bound of the order of the 
torsion subgroup of some CM abelian variety 
defined over a number field with values in full cyclotomic fields. 
(Here, a number field is a finite extension of $\mbb{Q}$.)

\begin{theorem}
\label{gsurf2}
Let $g>0$ be an integer.
Let $K$ be a number field of degree $d$ 
and denote by $h$  the narrow class number of $K$.
Let $K(\mu_{\infty})$ be the field 
obtained by adjoining to $K$ all roots of unity. 
Let $A$ be a $g$-dimensional abelian variety over $K$ with  complex multiplication
which has good reduction everywhere. 
Then, we have 
$$
 A(K(\mu_{\infty}))_{\mrm{tor}}\subset A[N]
$$
where 
$$
N:=\left( \prod_{p} p\right)^{2g^2\cdot (2g)!\cdot \Phi(g)H(g)\cdot dh
+12g^2-18g+10}.
$$
Here, $p$ ranges over the prime numbers  such that 
either 
$p\le (1+\sqrt{2}^{dh})^{2g}$ or 
$p$ is ramified in $K$.
\end{theorem}
We should note that Chou gave in \cite{Ch}
the complete list of the groups that appears as $A(\mbb{Q}(\mu_{\infty}))_{\mrm{tor}}$
as $A$ ranges over all elliptic curves defined over $\mbb{Q}$.
For CM elliptic curves $A$ over a  number field $K$, 
more precise observations for the order of $A(K(\mu_{\infty}))_{\mrm{tor}}$
than ours are studied in \cite{CCM}.

\begin{acknowledgments}
The author would like to thank  Yuichiro Taguchi 
for useful discussion and correspondence
to the proof of our main results.
The author would like to thank Manabu Yoshida 
for giving us various advice on  an earlier draft.
Thanks also are due to Takaichi Fujiwara for helpful advice
on  functions discussed in Section \ref{Phi-H}.
This work is supported by JSPS KAKENHI Grant Number JP19K03433. 
\end{acknowledgments}

\vspace{5mm}
\noindent
{\bf Notation :}
For any perfect field $F$, 
we denote by $G_F$  the absolute Galois group of $F$.
In this paper,  a $p$-adic field is a finite extension of $\mbb{Q}_p$.
If $F$ is an algebraic extension of $\mbb{Q}_p$, 
we  denote by $\cO_F$ and $\mbf{m}_F$ 
the ring of integers of $F$ and its maximal ideal, respectively. 
We also denote by $F^{\mrm{ab}}$ the maximal abelian extension 
of $F$ (in a fixed algebraic closure of $F$).
We put $d_F=[F:\mbb{Q}_p]$ if $F$ is a $p$-adic field. 
For an algebraic extension $F'/F$,
we denote by $e_{F'/F}$ and $f_{F'/F}$ 
the ramification index of $F'/F$ and 
the extension degree of the residue field extension of  $F'/F$,
respectively.
We set $e_F:=e_{F/\mbb{Q}_p}$ and $f_{F}:=f_{F/\mbb{Q}_p}$, 
and  also set $q_F:=p^{f_F}$. 
Finally, we denote by $\Gamma_F$ the set of $\mbb{Q}_p$-algebra 
embeddings of $F$ into 
a (fixed) algebraic closure $\overline{\mbb{Q}}_p$ of $\mbb{Q}_p$.

%%%%%%%%%%%%%%%%%%%%%%%%%%%%%%%%%%%%%%%%%%%%%%%%%%%%%%%%%%%%%%%%%%
%%%%%%%%%%%%%%%%%%%%%%%%%%%%%%%%%%%%%%%%%%%%%%%%%%%%%%%%%%%%%%%%%%
% Proof of Theorem
%%%%%%%%%%%%%%%%%%%%%%%%%%%%%%%%%%%%%%%%%%%%%%%%%%%%%%%%%%%%%%%%%%
%%%%%%%%%%%%%%%%%%%%%%%%%%%%%%%%%%%%%%%%%%%%%%%%%%%%%%%%%%%%%%%%%%

%\section{Preliminaries}
%
%In this section, we give some results on finiteness criteria of torsion of  algebraic groups
%and locally algebraic representations.
%
%\subsection{Finiteness criteria of torsion of  algebraic groups}

.

\section{Evaluations of some $p$-adic valuations for  characters}
\label{keysection}
We fix an algebraic closure 
$\overline{\mbb{Q}}_p$ of $\mbb{Q}_p$.
Throughout this section, 
\if0
\footnote{This assumption is needed 
to use the theory of locally algebraic representations.}
\fi 
we assume that all $p$-adic fields are subfields of $\overline{\mbb{Q}}_p$.
Denote by $v_p$ the $p$-adic valuation normalized by $v_p(p)=1$.
For  any  continuous character 
$\psi$ of $G_{K}$, we often regard $\psi$ as a character of $\mrm{Gal}(K^{\mrm{ab}}/K)$.
We denote by $\mrm{Art}_K$  
the local Artin map $K^{\times}\to \mrm{Gal}(K^{\mrm{ab}}/K)$
with arithmetic normalization.
We set
$\psi_{K}:=\psi\circ \mrm{Art}_{K}$. 
We denote by $\widehat{K}^{\times}$ 
the profinite completion of $K^{\times}$.
Note that  the local Artin map induces a topological isomorphism
$\mrm{Art}_K\colon \widehat{K}^{\times}\overset{\sim}{\rightarrow} \mrm{Gal}(K^{\mrm{ab}}/K)$.
For a uniformizer $\pi_K$ of  $K$ ,
we denote by $\chi_{\pi_K}\colon G_K\to \cO_K^{\times}$
the Lubin-Tate character associated with $\pi_K$.
By definition, the character $\chi_{\pi_K}$ is characterized by 
$\chi_{\pi_K,K}(\pi_K)=1$ and $\chi_{\pi_K,K}(x)=x^{-1}$ for any $x\in \cO_K^{\times}$.
Let $\pi$ be a uniformizer of $k$
and  denote by $k_{\pi}$ the Lubin-Tate extension of $k$ associated with $\pi$.
The field corresponding to 
the kernel  of the Lubin-Tate character $\chi_{\pi}\colon G_k\to \cO_k^{\times}$
is  $k_{\pi}$, 
and   $k_{\pi}$ is a totally ramified abelian extension of $k$.

\begin{proposition}
\label{vplem}
Let $\psi_1,\dots , \psi_n \colon G_{K}\to M^{\times}$  be continuous characters.
Then we have 
\begin{align*}
& \ \mrm{Min}\left\{ \sum^n_{i=1} v_p(\psi_i(\sigma)-1) \mid \sigma\in G_{Kk_{\pi}} \right\} 
\\
\le & \   
\mrm{Min}\left\{ \sum^n_{i=1} v_p(\psi_{i,Kk}(\omega)-1) 
\mid \omega\in \mrm{Nr}_{Kk/k}^{-1}(\pi^{f_{Kk/k}\mbb{Z}}) \right\}.
\end{align*}
\end{proposition}

\begin{proof}
This is Proposition 3 of \cite{Oz1} but we include a proof here
for the sake of completeness.
Let $M$ be the maximal unramified extension of $k$ contained in $Kk$.
The group $\mrm{Art}^{-1}_k(\mrm{Gal}(k^{\mrm{ab}}/M))$ contains 
$\mrm{Art}^{-1}_k(\mrm{Gal}(k^{\mrm{ab}}/k^{\mrm{ur}}))=\cO_k^{\times}$.
Furthermore,  $\mrm{Art}^{-1}_k(\mrm{Gal}(k^{\mrm{ab}}/M))$
is a subgroup of $\widehat{k}^{\times}=\pi^{\widehat{\mbb{Z}}}\times \cO_k^{\times}$
of index $[M:k]=f_{Kk/k}$.
Thus it holds that 
$\mrm{Art}^{-1}_k(\mrm{Gal}(k^{\mrm{ab}}/M))
=\pi^{f_{Kk/k}\widehat{\mbb{Z}}}\times \cO_k^{\times}$.
Since we have $\mrm{Art}^{-1}_k(\mrm{Gal}(k^{\mrm{ab}}/k_{\pi}))=\pi^{\widehat{\mbb{Z}}}$,
we obtain $\mrm{Art}^{-1}_k(\mrm{Gal}(k^{\mrm{ab}}/Mk_{\pi}))=\pi^{f_{Kk/k}\widehat{\mbb{Z}}}$.
If we denote by $\mrm{Res}_{Kk/k}$ the natural restriction map from 
$\mrm{Gal}((Kk)^{\mrm{ab}}/Kk)$ to $\mrm{Gal}(k^{\mrm{ab}}/k)$, 
it is not difficult to check 
$\mrm{Res}_{Kk/k}^{-1}(\mrm{Gal}(k^{\mrm{ab}}/Mk_{\pi}))
=\mrm{Gal}((Kk)^{\mrm{ab}}/Kk_{\pi})$.
Thus we find 
$\mrm{Art}^{-1}_{Kk}(\mrm{Gal}((Kk)^{\mrm{ab}}/Kk_{\pi}))=\mrm{Nr}_{Kk/k}^{-1}(\pi^{f_{Kk/k}\widehat{\mbb{Z}}})$.
Now the lemma follows from 
\begin{align*}
& \ \mrm{Min}\left\{ \sum^n_{i=1} v_p(\psi_i(\sigma)-1) \mid \sigma\in G_{Kk_{\pi}} \right\}  \\
= & \
\mrm{Min}\left\{ \sum^n_{i=1} v_p(\psi_{i, Kk}\circ \mrm{Art}^{-1}_{Kk}(\sigma)-1) 
\mid \sigma\in \mrm{Gal}((Kk)^{\mrm{ab}}/Kk_{\pi}) \right\}.
\end{align*}
\end{proof}

We recall an observation of Conrad.
We denote by $K_0$ 
the maximal unramified extension of  $\mbb{Q}_p$ contained in $K$
and set $D^K_{\mrm{cris}}(\ast) := (B_{\mrm{cris}}\otimes_{\mbb{Q}_p} \ast)^{G_K}$. 
We denote by $\underline{K}^{\times}$ 
the Weil restriction  $\mrm{Res}_{K/\mbb{Q}_p}(\mbb{G}_m)$.
\begin{proposition}
\label{LA2}
Let $\psi\colon G_K\to M^{\times}$ be 
a continuous character. 

\noindent
{\rm (1)} $M(\psi)$ is crystalline if and  only if 
there exists a {\rm(}necessarily unique{\rm )} $\mbb{Q}_p$-homomorphism 
$\psi_{\rm alg}\colon \underline{K}^{\times}\to \underline{M}^{\times}$
such that $\psi_K$ and $\psi_{\rm alg}$ {\rm (}on $\mbb{Q}_p$-points{\rm )} 
coincides on $\cO_K^{\times} (\subset \underline{K}^{\times}(\mbb{Q}_p))$.

\noindent
{\rm (2)} 
Assume that $M(\psi)$ is crystalline and let $\psi_{\rm alg}$ be as in {\rm (1)}. 
{\rm (}Note that $M(\psi^{-1})$ is also  crystalline.{\rm )}
Then, the filtered $\vphi$-module 
$D^K_{\mrm{cris}}(M(\psi^{-1}))=(B_{\mrm{cris}}\otimes_{\mbb{Q}_p} M(\psi^{-1}))^{G_K}$
over $K$ is free of rank $1$ over $K_0\otimes_{\mbb{Q}_p} M$ 
and its $K_0$-linear endomorphism $\vphi^{f_K}$ 
is given by the action of the product 
$\psi_K(\pi_K)\cdot \psi_{\rm alg}^{-1}(\pi_K)\in M^{\times}$.  
Here, $\pi_K$ is any uniformizer of $K$.
\end{proposition}

\begin{proof}
This is Proposition B.4 of \cite{Co}.
\end{proof}

Let $\psi\colon G_K\to M^{\times}$ be a crystalline character. 
For any $\sigma\in \Gamma_M$, let $\chi_{\sigma M}\colon I_{\sigma M}\to \sigma M^{\times}$ 
be the restriction to the inertia $ I_{\sigma M}$ of the Lubin-Tate character associated  
with any choice of uniformizer of $\sigma M$ 
(it depends on the choice of a uniformizer of $\sigma M$,
but its restriction to the inertia subgroup does not).
Assume that $K$ contains the Galois Closure of $M/\mbb{Q}_p$. 
Then, we have 
$$
\psi = 
\prod_{\sigma \in \Gamma_M} \sigma^{-1} \circ \chi_{\sigma M}^{h_{\sigma}}
$$
on the inertia $I_K$ for some integer $h_{\sigma}$. 
Equivalently,  the character
$\psi_{\rm alg}$ on $\mbb{Q}_p$-points  coincides with 
$\prod_{\sigma \in \Gamma_M} \sigma^{-1} \circ \mrm{Nr}_{K/\sigma M}^{-h_{\sigma}}$. 
Note that $\{h_{\sigma} \mid \sigma\in \Gamma_M \}$ 
is the set of Hodge-Tate weights of $M(\psi)$, 
that is, $C\otimes_{\mbb{Q}_p} M(\psi)\simeq \oplus_{\sigma \in \Gamma_M} C(h_{\sigma})$
where $C$ is the completion of $\overline{\mbb{Q}}_p$.

For integers $d,h$ and a $p$-adic field $M$, 
we define a constant $C(d,M,h)$ by 
\begin{equation}
\label{const0}
C(d,M,h) := \ v_p(d/d_M) + h +  \frac{d_M}{2}\left(d_M+v_p(e_M)-\frac{1}{e_M}+v_p(2)(d_M-1)\right).
\end{equation}

\begin{theorem}
\label{key:est2}
Let $\psi_1,\dots ,\psi_n\colon G_K\to M^{\times}$ be crystalline characters
and $h\ge 0$ an integer.
Assume that $M$ is a Galois extension of $\mbb{Q}_p$ and 
$K$ contains $M$.
Suppose that, for each $i$, we have 
$$
\psi_i= 
\prod_{\sigma \in \Gamma_M} \sigma^{-1} \circ \chi_{M}^{h_{i,\sigma}}
$$
on the inertia $I_K$; thus $\{h_{i,\sigma} \mid \sigma\in \Gamma_M \}$ 
is the Hodge-Tate weights of $M(\psi_i)$. 
We assume the following conditions.
\begin{itemize}
\item[{\rm (i)}]  $\{h_{i, \sigma} \mid \sigma\in \Gamma_M \}$ contains 
at least two different integers for each $i$. 
{\rm (}In particular, we have $M\not=\mbb{Q}_p$.{\rm )}
\item[{\rm (ii)}] We have $\mrm{Min}\left\{ v_p(h_{i,\sigma}-h_{i,\tau}) \mid \sigma,\tau\in \Gamma_M \right\} \le h$
for each $i$.
\end{itemize}
%Put $\hat{\psi}_{i,K}:=\psi_{i,K}\circ \mrm{Nr}_{Kk/K}$. 

\vspace{2mm}

\noindent
{\rm (1)} There exists an element $\hat{\omega}\in \mrm{ker}\ \mrm{Nr}_{M/\mbb{Q}_p}$
with the property that, for every $1\le i\le n$, it holds that 
\begin{equation}
\label{eval'}
1+v_p(2)\le v_p(\psi_{i,K}(\hat\omega)^{-1}-1) \le \delta_{(i)}+C(d_{K},M,h).
\end{equation}
Here, 
$$
\delta_{(i)}:= \left\{
\begin{array}{cl}
0 &\quad 
\mbox{if $i=1,2$}, \cr
2i-5 &\quad 
\mbox{if $i\ge 3$}.
\end{array}
\right.
$$

\noindent
{\rm (2)} Let $\hat\omega$ be as in {\rm (1)}.
For any $x\in K^{\times}$, 
there exists an integer $0\le s(x)\le n$ with the property 
that, for every $1\le i \le n$, it holds that 
\begin{equation}
\label{nonunitcase'}
v_p(\psi_{i,K}(x\hat\omega^{p^{s(x)}})^{-1}-1)
\le 
n+\delta_{(i)}+C(d_{K},M,h).
\end{equation} 
\end{theorem}

\begin{proof}
Take an element $x\in \cO_M$ such that $\cO_M=\mbb{Z}_p[x]$.
We set $p':=p$ or $p':=4$ if $p\not=2$ or $p=2$, respectively,
and  put $x'=p'x$.
Set $m^{\tau}_{r, \sigma}:= d_{K/M}(h_{r, \tau\sigma}-h_{r, \sigma})$
for $1\le r\le n$ and $\sigma,\tau \in \Gamma_M$.
We also set 
\begin{equation*}
y^{\tau}_{r,\ell}  
:=\sum_{\sigma\in \Gamma_M} m^{\tau}_{r, \sigma} (\sigma^{-1}x')^{\ell-1} 
\end{equation*} 
for  $1\le \ell\le d_M$. 
(Note that $y^{\tau}_{r,1}=0$.)
Set
\begin{equation*}
\omega_{\ell}:=\mrm{exp}((x')^{\ell-1})\quad  \mbox{and} \quad  
\omega^{\tau}_{\ell}:=\frac{\tau \omega_{\ell}}{\omega_{\ell}}
\end{equation*}
for any  $1\le \ell\le d_M$ and $\tau\in \Gamma_M$.
It holds $\omega^{\tau}_{\ell}\in \ker \mrm{Nr}_{M/\mbb{Q}_p}$
by construction.
\begin{lemma}
\label{exp1}
We have $\mrm{exp}(y^{\tau}_{r,\ell})=\psi_{r,K}(\omega^{\tau}_{\ell})^{-1}$.
\end{lemma}
\begin{proof}
We see 
$$
\psi_{r,K}(\omega_{\ell})^{-1}
=\prod_{\sigma \in \Gamma_M} 
\sigma^{-1} \circ \mrm{Nr}_{K/M}(\omega_{\ell})^{h_{r, \sigma}}
=\left(\prod_{\sigma \in \Gamma_M} 
\sigma^{-1}\omega_{\ell}^{h_{r, \sigma}}\right)^{d_{K/M}}.
$$
We also have $\psi_{r,K}(\tau\omega_{\ell})^{-1}
=\left(\prod_{\sigma \in \Gamma_M} 
\sigma^{-1}\tau\omega_{\ell}^{h_{r, \sigma}}\right)^{d_{K/M}}
=\left(\prod_{\sigma \in \Gamma_M} 
\sigma^{-1}\omega_{\ell}^{h_{r, \tau\sigma}}\right)^{d_{K/M}}$.
Thus we have
$$
\psi_{r,K}(\omega^{\tau}_{\ell})^{-1}
=\left(\prod_{\sigma \in \Gamma_M} 
\sigma^{-1}\omega_{\ell}^{h_{r, \tau\sigma}-h_{r, \sigma}}\right)^{d_{K/M}}
=\prod_{\sigma\in \Gamma_M} \sigma^{-1} \omega_{\ell}^{m^{\tau}_{r, \sigma}}.
$$
On the other hand, we have 
$$
\mrm{exp}(y^{\tau}_{r,\ell})
=\mrm{exp}\left( \sum_{\sigma\in \Gamma_M} 
m^{\tau}_{r, \sigma} (\sigma^{-1}x')^{\ell-1} \right)
=\prod_{\sigma\in \Gamma_M} \mrm{exp}((\sigma^{-1}x')^{\ell-1})^{m^{\tau}_{r, \sigma}}
= \prod_{\sigma\in \Gamma_M} \sigma^{-1} \omega_{\ell}^{m^{\tau}_{r, \sigma}}.
$$
Thus we obtain the lemma.
\end{proof}
We furthermore need the following
evaluation.
%%%%%%%%%%%%%%%%%%%%%%%%%%%%%%%%%%%%%%%%%%%%%%%%%%%%%%%%%%%%%%%%%
%%%%%%%%%%%%%%%%%%%%%%%%%%%%%%%%%%%%%%%%%%%%%%%%%%%%%%%%%%%%%%%%%
%%%%%%%%%%%%%%%%%%%%%%%%%%%%%%%%%%%%%%%%%%%%%%%%%%%%%%%%%%%%%%%%%
\begin{lemma}
\label{Lem:bound}
For each $1\le r\le n$,  
there exist $\tau_r\in \Gamma_M$ and  an integer $2\le \ell_r \le d_M$
such that 
\begin{equation*}
v_p(y^{\tau_r}_{r,\ell_r})\le C(d_{K},M,h).
\end{equation*}
\end{lemma}
%%%%%%%%%%%%%%%%%%%%%%%%%%%%%%%%%%%%%%%%%%%%%%%%%%%%%%%%%%%%%%%%%
%%%%%%%%%%%%%%%%%%%%%%%%%%%%%%%%%%%%%%%%%%%%%%%%%%%%%%%%%%%%%%%%%
%%%%%%%%%%%%%%%%%%%%%%%%%%%%%%%%%%%%%%%%%%%%%%%%%%%%%%%%%%%%%%%%%
\begin{proof}
In this proof, we fix $r$. 
By the assumption (i), 
there exist $\tau_1,\tau_2\in \Gamma_M$ such that $h_{r,\tau_1}\not=h_{r,\tau_2}$.
We choose such $\tau_1$  and $\tau_2$ so that $v_p(h_{r,\tau_1}-h_{r,\tau_2})
= \mrm{Min}\left\{ v_p(h_{r,\sigma}-h_{r,\tau}) \mid \sigma,\tau\in \Gamma_M \right\} $.
Set $\tau:=\tau_2\tau_1^{-1}\in \Gamma_{M}$.
We write $\Gamma_{M}=\{\tau_1,\tau_2,\dots ,\tau_{d_M}\}$.
Note that $m^{\tau}_{r,\tau_1}=d_{K/M}(h_{r,\tau_2}-h_{r,\tau_1})$ is not zero.
We denote by $X\in M_d(\cO_M)$ the matrix whose $(i,j)$-component 
is $(\tau_i^{-1}x')^{j-1}$.
Then we have 
\begin{equation}
\label{matrix}
\begin{pmatrix}
y^{\tau}_{r,1} & \cdots & y^{\tau}_{r,d_M}  
\end{pmatrix}
=
\begin{pmatrix}
m^{\tau}_{r,\tau_1} & \cdots & m^{\tau}_{r,\tau_{d_M}}  
\end{pmatrix}
X
\end{equation}
and the determinant $\det X$ of $X$ is 
$\prod_{1\le i< j\le d_M} (\tau_j^{-1}x'-\tau_i^{-1}x') 
=(p')^{\frac{d_M(d_M-1)}{2}} \prod_{1\le i< j\le  d_M} 
(\tau_j^{-1}x-\tau_i^{-1}x)$.
We also have 
\begin{align*}
v_p\left(\prod_{1\le i< j\le d_M} 
(\tau_j^{-1}x-\tau_i^{-1}x)\right)
& = \sum_{1\le i< j\le d_M}  v_p\left(\tau_j^{-1}x-\tau_i^{-1}x\right) \\
& = \frac{1}{2}\sum_{1\le i, j\le d_M, i\not=j}  
v_p\left(\tau_j^{-1}x-\tau_i^{-1}x\right)\\
& = \frac{d_M}{2}v_p(\mcal{D}_{M/\mbb{Q}_p})
\le  \frac{d_M}{2}\left(1+v_p(e_{M})-\frac{1}{e_{M}}\right).
\end{align*}
(cf.\ \cite[Chapter 3, Section 6, Proposition 13]{Se1}),\ 
where 
$\mcal{D}_{M/\mbb{Q}_p}$ is the differential of $M/\mbb{Q}_p$.
We find 
\begin{equation}
\label{detX}
v_p(\det X)\le 
\frac{d_M}{2}\left(d_M+v_p(e_{M})-\frac{1}{e_{M}}
+v_p(2)(d_M-1)\right).
\end{equation}
By \eqref{matrix},  
we have $m^{\tau}_{r,\tau_1}\det X
=\sum^{d_M}_{\ell=1}y^{\tau}_{r,\ell}x_{\ell}$ for some $x_{\ell}\in \cO_M$,
which gives the fact that 
there exists an integer  $\ell_r=\ell$ with the property that 
$v_p(y^{\tau}_{r,\ell})\le v_p(m^{\tau}_{r,\tau_1}\det X)$. By \eqref{detX},
we have 
$$
v_p(y^{\tau}_{r,\ell}) \le  
v_p(d_{K/M})+v_p(h_{r,\tau_1}-h_{r,\tau_2})+v_p(\det X)\le C(d_{K},M,h)
$$
as desired.
We remark that $\ell$ is not equal to $1$ since 
$y^{\tau}_{r,1}$ is zero. 
\end{proof}

Now we return to the proof of Theorem \ref{key:est2}.
Take $\tau_r$ and  $\ell_r$ as in Lemma \ref{Lem:bound} with 
the additional condition that  
\begin{equation}
\label{y'}
v_p(y^{\tau_r}_{r,\ell_r})
=\mrm{Min}\{ v_p(y^{\tau}_{r,\ell})  \mid \tau\in \Gamma_M, 2\le \ell\le d_M\}. 
\end{equation}
Here we consider an element  $\hat\omega\in \ker \mrm{Nr}_{M/\mbb{Q}_p}$ 
which is of the form 
$\hat\omega =\prod^{n}_{r=1} (\omega^{\tau_r}_{\ell_r})^{s_r}$,
where $s_r$ is defined inductively by the following.
$$
(s_1,s_2)= \left\{
\begin{array}{cl}
(0,1) &\quad 
\mbox{if $v_p(y^{\tau_1}_{1,\ell_1})= v_p(y^{\tau_2}_{1,\ell_2})$}, \cr
(1,0) &\quad \mbox{if $v_p(y^{\tau_1}_{1,\ell_1})\not= v_p(y^{\tau_2}_{1,\ell_2})$ and 
$v_p(y^{\tau_1}_{2,\ell_1}) = v_p(y^{\tau_2}_{2,\ell_2})$}, \cr 
(1,1) &\quad 
\mbox{if $v_p(y^{\tau_1}_{1,\ell_1})\not= v_p(y^{\tau_2}_{1,\ell_2})$
and $v_p(y^{\tau_1}_{2,\ell_1})\not= v_p(y^{\tau_2}_{2,\ell_2})$.}
\end{array}
\right.
$$

$$
s_3= \left\{
\begin{array}{cl}
p &\quad 
\mbox{if $v_p(s_1y^{\tau_1}_{3,\ell_1}+s_2y^{\tau_2}_{3,\ell_2})\not= v_p(py^{\tau_3}_{3,\ell_3})$}, \cr
p^2 &\quad \mbox{if $v_p(s_1y^{\tau_1}_{3,\ell_1}+s_2y^{\tau_2}_{3,\ell_2})= v_p(py^{\tau_3}_{3,\ell_3})$.}
\end{array}
\right.
$$
For $r\ge 4$, 
$$
s_r= \left\{
\begin{array}{cl}
ps_{r-1} &\quad 
\mbox{if $v_p(\sum^{r-1}_{j=1} s_j y^{\tau_j}_{r,\ell_j})\not= v_p(ps_{r-1}y^{\tau_r}_{r,\ell_r})$}, \cr
p^2s_{r-1} &\quad 
\mbox{if $v_p(\sum^{r-1}_{j=1} s_j y^{\tau_j}_{r,\ell_j})= v_p(ps_{r-1}y^{\tau_r}_{r,\ell_r})$}.
\end{array}
\right.
$$
We calim that we have
$$
1+v_p(2)\le v_p\left(\sum^{n}_{r=1} s_ry^{\tau_r}_{i,\ell_r}\right) \le \delta_{(i)}+C(d_{K},M,h)
$$
for any $i$, where 
$\delta_{(i)}$ is as in the statement (1).  
The inequality $1+v_p(2)\le v_p\left(\sum^{n}_{r=1} s_ry^{\tau_r}_{i,\ell_r}\right)$
is clear since we always have  $1+v_p(2)\le v_p(y^{\tau}_{i,\ell})$
by definition of $y^{\tau}_{i,\ell}$.
We show $v_p\left(\sum^{n}_{r=1} s_ry^{\tau_r}_{i,\ell_r}\right) \le \delta_{(i)}+C(d_{K},M,h)$
by induction on $i$.
\begin{itemize}
\item Suppose either $i=1$ or $i=2$.
By  \eqref{y'}  and the inequality 
$0< v_p(s_r)$ for $r\ge 3$,
it is not difficult to check  
$v_p\left(\sum^{n}_{r=1} s_ry^{\tau_r}_{i,\ell_r}\right)
=v_p(y^{\tau_i}_{i,\ell_i})$.
Furthermore, we have 
$v_p(y^{\tau_i}_{i,\ell_i}) 
\le C(d_K,M,h)= \delta_{(i)}+C(d_K,M,h)$ by Lemma \ref{Lem:bound}.

\item Suppose $i\ge 3$. By definition of $s_i$
we have 
$v_p\left(\sum^{i-1}_{r=1} s_ry^{\tau_r}_{i,\ell_r}\right)\not= v_p(s_iy^{\tau_i}_{i,\ell_i})$.
We also have 
$v_p\left(\sum^{n}_{r=i} s_ry^{\tau_r}_{i,\ell_r}\right)
=v_p(s_iy^{\tau_i}_{i,\ell_i})$
since $v_p(s_i y^{\tau_i}_{i,\ell_i})<v_p(s_r y^{\tau_r}_{i,\ell_r})$ for $i<r$.
Hence, it follows from Lemma \ref{Lem:bound} that 
we have  
\begin{align*}
v_p\left(\sum^{n}_{r=1} s_ry^{\tau_r}_{i,\ell_r}\right)
& = \mrm{Min}\left\{ v_p\left(\sum^{i-1}_{r=1} s_ry^{\tau_r}_{i,\ell_r}\right),  
v_p(s_iy^{\tau_i}_{i,\ell_i}) \right\} \\
& \le  v_p(ps_{i-1}y^{\tau_i}_{i,\ell_i})
\le 1+v_p(s_{i-1})+C(d_K,M,h)
\end{align*}
if $i\ge 4$.
Since we have $v_p(s_{i-1})\le 2(i-3)$ if $i\ge 4$,
the  claim  for $i\ge 4$ follows.
The claim for  $i=3$ follows by a similar manner; 
we have  
$v_p\left(\sum^{n}_{r=1} s_ry^{\tau_r}_{3,\ell_r}\right)
\le  v_p(py^{\tau_3}_{3,\ell_3})
\le 1+C(d_K,M,h)
=\delta_{(3)}+C(d_K,M,h)$.
\end{itemize}
By  construction of $\hat\omega$ and Lemma \ref{exp1},
we see
$$
\psi_{i,K}(\hat\omega)^{-1}
=\prod^{n}_{r=1} \psi_{i,K}(\omega^{\tau_r}_{\ell_r})^{-s_r}
=\prod^{n}_{r=1} \mrm{exp}\left(s_ry^{\tau_r}_{i,\ell_r}\right)
=\mrm{exp}\left(\sum^{n}_{r=1} s_ry^{\tau_r}_{i,\ell_r}\right).
$$
Thus we find 
$v_p(\psi_{i,K}(\hat\omega)^{-1}-1)=
v_p\left(\sum^{n}_{r=1} s_ry^{\tau_r}_{i,\ell_r}\right).$
Therefore, the claim above gives the statement   
(1) of Theorem \eqref{key:est2}.

We show (2).
We set $m_i:=\psi_{i,K}(x)^{-1}-1$ 
and $\theta^{(s)}_i=\psi_{i,K}(\hat\omega^{p^s})^{-1}-1$
for any $s\ge 0$.
It follows from the condition 
$v_p(\psi_{i,K}(\hat\omega)^{-1}-1)\ge 1+v_p(2)$ that 
the equality $v_p(\theta^{(s)}_i)=s+v_p(\theta_i^{(0)})$ holds.
For each $1\le i\le n$,
there exists at most only one integer $s\ge 0$ so that 
$v_p(m_i)=v_p(\theta_i^{(s)})$ 
since $\{v_p(\theta_i^{(s)})\}_s$ is  strictly increasing.
Hence, there exists an integer $0\le s(x)\le n$
with the property that 
$v_p(m_i)\not=v_p(\theta_i^{(s(x))})$ for every $1\le i \le n$
(by "Pigeonhole principle").
With this choice of $s(x)$, we obtain 
$v_p(\psi_{i,K}(x\hat\omega^{p^{s(x)}})^{-1}-1)
=v_p(m_i+\theta_i^{(s(x))}+m_i\theta_i^{(s(x))})\le v_p(\theta_i^{(n)})
=n+v_p(\theta_i^{(0)}).
$
This finishes the proof of (2).
\end{proof}

\if0
\subsection{Some  results on  abelian varieties}

The following is well-known but it will be useful to write down here.
\begin{proposition}
\label{WB}
Let $K$ be a $p$-adic field and $\mbb{F}$ the residue field of $K(\mu_{p^{\infty}})$.

\noindent
{\rm (1)} The order of $\mbb{F}$ is at most $p^{d_K}$.
 
\noindent
{\rm (2)} For  a $g$-dimensional abelian variety  $A$ over $\mbb{F}_K$, 
we have $v_p(\sharp A(\mbb{F}))\le \log_p(1+p^{\frac{d_K}{2}})^{2g}$.
\end{proposition}

\begin{proof}
The assertion (1) follows from the fact that 
$f_{K(\mu_{p^{\infty}})}
=f_{K(\mu_{p^{\infty}})/\mbb{Q}_p(\mu_{p^{\infty}})}$
divides $d_K$.
If we denote by $q$ the order of $\mbb{F}$, 
it follows from the Weil bound that 
$\sharp A(\mbb{F})
\le (1+\sqrt{q})^{2g}
\le (1+p^{\frac{d_K}{2}})^{2g}
$.
This in particular implies  (2).
\end{proof}
\fi

%%%%%%%%%%%%%%%%%%%%%%%%%%%%%%%%%%%%%%%%%%%%%%%%%%%%%%%%%%%%%%%%%
%%%%%%%%%%%%%%%%%%%%%%%%%%%%%%%%%%%%%%%%%%%%%%%%%%%%%%%%%%%%%%%%%
%%%%%%%%%%%%%%%%%%%%%%%%%%%%%%%%%%%%%%%%%%%%%%%%%%%%%%%%%%%%%%%%%
%%%%%%%%%%%%%%%%%%%%%%%%%%%%%%%%%%%%%%%%%%%%%%%%%%%%%%%%%%%%%%%%%
%%%%%%%%%%%%%%%%%%%%%%%%%%%%%%%%%%%%%%%%%%%%%%%%%%%%%%%%%%%%%%%%%
%%%%%%%%%%%%%%%%%%%%%%%%%%%%%%%%%%%%%%%%%%%%%%%%%%%%%%%%%%%%%%%%%
%%%%%%%%%%%%%%%%%%%%%%%%%%%%%%%%%%%%%%%%%%%%%%%%%%%%%%%%%%%%%%%%%
%%%%%%%%%%%%%%%%%%%%%%%%%%%%%%%%%%%%%%%%%%%%%%%%%%%%%%%%%%%%%%%%%
%%%%%%%%%%%%%%%%%%%%%%%%%%%%%%%%%%%%%%%%%%%%%%%%%%%%%%%%%%%%%%%%%
\section{Proof of main theorems}
\label{Mainpart}

The main purpose of this section is to show Theorems \ref{MT:CM}
and \ref{MT:CM:KT} in Introduction.  
As for Theorem \ref{MT:CM}, 
we show a slightly refined statement as follows.
\begin{theorem}
\label{MT:CM:refined}
Let $g>0$ be a positive integer. 
Let  $k$ be a $p$-adic field  with residue cardinality $q_k$ and  $\pi$ a uniformizer of $k$.
Put $p'=p$ or $p'=4$ if $p\not=2$ or $p=2$, respectively.
Let $\mu\ge 1$ be the smallest integer\footnote{If 
$q_k^{-1}\mrm{Nr}_{k/\mbb{Q}_p}(\pi)$ 
is a root of unity, the constant $\mu$ here coincides with $\mu$ 
appeared in Theorem \ref{MT:CM}.} so that
$$
(q_k^{-1}\mrm{Nr}_{k/\mbb{Q}_p}(\pi))^{\mu}\equiv 1 \ \mrm{mod}\ p'.
$$
Assume the following conditions\footnote{The condition (i) here 
depends on the choice of $K$. However, 
the author hopes that this condition would be replaced with
certain one which does not depend on $K$
as (i) in Theorem \ref{MT:CM}.}.
\begin{itemize}
\item[{\rm (i)}] $v_p((q_k^{-1}\mrm{Nr}_{k/\mbb{Q}_p}(\pi))^{\mu}-1)
>g\cdot (2g)!\cdot \Phi(g)H(g) \cdot \mu \cdot d_{Kk/k}f_k$
and 
\item[{\rm (ii)}] $d_k$ is prime to $(2g)!$.
\end{itemize}
Then, for any $g$-dimensional abelian variety $A$ over 
a $p$-adic field $K$ with complex multiplication, 
we have 
$$
A(Kk_{\pi})[p^{\infty}]\subset A[p^C]
$$
where 
$$
C := 2g^2\cdot (2g)!\cdot \Phi(g)H(g)\cdot \mu \cdot d_{Kk}  
+12g^2-18g+10.
$$
In particular, we have 
$$
\sharp A(Kk_{\pi})[p^{\infty}]\le p^{2gC}.
$$
\end{theorem}

Our proofs of Theorems \ref{MT:CM:refined} and \ref{MT:CM:KT} 
proceed by similar methods.
As in the previous section, 
we fix an algebraic closure 
$\overline{\mbb{Q}}_p$ of $\mbb{Q}_p$
and suppose that $K$ is a subfield of $\overline{\mbb{Q}}_p$.
In this section, we often use the following technical constants:
\begin{align*}
& L_g(m):= \left[\log_p(1+p^{\frac{m}{2}})^{2g}\right],\\ 
& C(m,M,h):=v_p\left( \frac{m}{d_M} \right)  + h 
+ \frac{d_M}{2}\left(d_M+v_p(e_M)-\frac{1}{e_M} +v_p(2)(d_M-1)\right).
\end{align*}
Here, $m\ge 1$ and $h\ge 0$ are integers and $M$ is a $p$-adic field.
\begin{remark}
\label{Lg}
(1) We have $mg\le L_g(m)<g(m+1+v_p(2))$ for any prime $p$ and $m\ge 1$, 
and we also have $L_g(m)<g(m+1)$ if $(p,m)\not = (2,1), (2,2)$. 

\noindent
(2) Moreover, we have\footnote{The evaluation $8g$ here is "rough" but it is 
enough for our proofs.} 
$$
L_g(m)=mg\quad \mbox{for  $m\ge 8g$}.
$$ 
This can be checked as follows:  
It suffices to show $(1+p^{\frac{m}{2}})^{2g}<p^{mg+1}$ for $m\ge 8g$.
This inequality is equivalent to $(1+p^{-\frac{m}{2}})^{2g}<p$.
Thus it is  enough to show $(1+2^{-\frac{m_0}{2}})^{2g}<2$ where $m_0:=8g$.
By inequalities $2g<2^{2g}$ and 
$\left(\begin{smallmatrix}
  2g \\ r
\end{smallmatrix}\right)<2^{2g}$
for $0\le r\le 2g$,
we find
\if0
\begin{align*}
\left(1+2^{-\frac{m_0}{2}}\right)^{2g}
& 
= \left(1+\left(\frac{1}{2}\right)^{\frac{m_0}{2}}\right)^{2g}
=1+\sum^{2g}_{r=1} 
\left(\begin{smallmatrix}
  2g \\ r
\end{smallmatrix}\right)
\left(\frac{1}{2}\right)^{\frac{rm_0}{2}} \\
& 
< 1+ 2g\cdot 2^{2g} \left(\frac{1}{2}\right)^{\frac{m_0}{2}}
< 1+ \left(\frac{1}{2}\right)^{\frac{m_0}{2}-4g}=2
\end{align*}
\fi
$$
\left(1+2^{-\frac{m_0}{2}}\right)^{2g}
=1+\sum^{2g}_{r=1} 
\left(\begin{smallmatrix}
  2g \\ r
\end{smallmatrix}\right)
\left(\frac{1}{2}\right)^{\frac{rm_0}{2}} 
< 1+ 2g\cdot 2^{2g} \left(\frac{1}{2}\right)^{\frac{m_0}{2}}
< 1+ \left(\frac{1}{2}\right)^{\frac{m_0}{2}-4g}=2
$$
as desired.
\end{remark}

%%%%%%%%%%%%%%%%%%%%%%%%%%%%%%%%%%%%%%%%%%%%%%%%%%%%%%%%%
%%%%%%%%%%%%%%%%%%%%%%%%%%%%%%%%%%%%%%%%%%%%%%%%%%%%%%%%%%
%%%%%%%%%%%%%%%%%%%%%%%%%%%%%%%%%%%%%%%%%%%%%%%%%%%%%%%%%%

\subsection{Special cases}

We consider Theorem \ref{MT:CM:refined} under some additional hypothesis. 
In this section, we show 
\begin{proposition}
\label{MT:CM:good}
Let the situation be as in Theorem \ref{MT:CM:refined}
except assuming not {\rm (i)} but 
\begin{itemize}
\item[{\rm (i)'}] $v_p((q_k^{-1}\mrm{Nr}_{k/\mbb{Q}_p}(\pi))^{\mu}-1)
>L_g((2g)!\cdot \mu \cdot d_{Kk/k}f_k)$.
\end{itemize}
Moreover, we assume that 
$A$ has good reduction over $K$ 
and all the endomorphisms of $A$
are defined over $K$.
Put
\begin{align*}
C_g(K,k) & =v_p(d_{Kk})+ \frac{(2g)!}{2} \left((2g)!+v_p((2g)!) + v_p(2)((2g)!-1)\right), \\
\Delta_g(K,k)& =\mrm{Max}\left\{ C_g(K,k), L_g((2g)!\cdot \mu \cdot d_{Kk/k}f_k) \right\}.
\end{align*}
Then,
we have 
$$
A(Kk_{\pi})[p^{\infty}] \subset A[p^C]
$$
where 
$$
C:= 2g\Delta_g(K,k) + 12g^2-18g+10.
$$
\end{proposition}

\begin{proof}
Put $T=T_p(A)$ and $V=V_p(A)$ to simplify notation.
Let $\rho\colon G_K\to GL_{\mbb{Z}_p}(T)$
be  the continuous homomorphism obtained by the $G_K$-action on $T$.
Fix an  isomorphism
$\iota \colon T\overset{\sim}{\rightarrow} \mbb{Z}_p^{\oplus 2g}$
of $\mbb{Z}_p$-modules. 
We have an isomorphism
$\hat{\iota}\colon GL_{\mbb{Z}_p}(T)\simeq GL_{2g}(\mbb{Z}_p)$ 
relative to $\iota$. 
We abuse notation by writing $\rho$
for the composite map 
$G_K\to GL_{\mbb{Z}_p}(T)
\simeq GL_{2g}(\mbb{Z}_p)$
of $\rho$ and $\hat{\iota}$.
Now let $P\in T$ and denote by $\bar{P}$ the image of $P$ in $T/p^nT$.
By definition, we have 
 $\iota(\sigma P)=\rho(\sigma)\iota(P)$
for $\sigma\in G_K$. 
Suppose that $\bar{P}\in(T/p^nT)^{G_{Kk_{\pi}}}$.
This implies  $\sigma P-P \in p^nT$ for any $\sigma\in {G_{Kk_{\pi}}}$. 
This is equivalent to say that 
 $(\rho(\sigma)-E)\iota(P) \in p^n\mbb{Z}_p^{\oplus 2g}$, and this in particular 
implies 
$\det (\rho(\sigma)-E)\iota(P) \in p^n\mbb{Z}_p^{\oplus 2g}$
for any $\sigma\in {G_{G_{Kk_{\pi}}}}$.
Hence we find 
$\det (\rho(\sigma)-E)P \in p^nT$
for any $\sigma\in {G_{Kk_{\pi}}}$.
Put 
$$
c=\mrm{Min}\{v_p(\det (\rho(\sigma)-E)))\mid \sigma\in {G_{Kk_{\pi}}}\}.
$$
Then we see $P\in p^{n-c}T$ (if $c$ is finite and $n>c$)  and  this shows 
$(T/p^nT)^{G_{Kk_{\pi}}}\subset p^{n-c}T/p^nT$.
This implies an inequality
\begin{equation}
\label{order:det}
A(Kk_{\pi})[p^{\infty}] \subset A[p^c]
\end{equation}
if $c$ is finite.

On the other hand, let us denote by $F$ the field of complex multiplication of $A$. 
We know that $V$ is a free $F\otimes_{\mbb{Q}} \mbb{Q}_p$-module of rank one
and the $G_K$-action on $V$ commutes with 
$F\otimes_{\mbb{Q}} \mbb{Q}_p$-action.
Let $\prod^{n}_{i=1} F_i$ denote the decomposition 
of $F\otimes_{\mbb{Q}} \mbb{Q}_p$ into a finite product of $p$-adic fields. 
This induces a decomposition $V\simeq \oplus^{n}_{i=1} V_i$
of $\mbb{Q}_p[G_K]$-modules. 
Each $V_i$ is equipped with a structure of one dimensional $F_i$-modules
and the $G_K$-action on $V_i$ commutes with $F_i$-action.
Let $\rho_i\colon G_K\to GL_{\mbb{Q}_p}(V_i)$ be the homomorphism
obtained by the $G_K$-action on $V_i$.
Since $\rho_i$ is abelian, 
it follows from the Shur's lemma that 
we have $(V_i\otimes_{\mbb{Q}_p} \overline{\mbb{Q}}_p)^{\mrm{ss}}\simeq 
\oplus^{d_{F_i}}_{j=1} \overline{\mbb{Q}}_p(\psi_{i,j})$
for some continuous characters $\psi_{i,j}\colon G_K\to \overline{\mbb{Q}}_p^{\times}$. 
Here, the subscript "ss" stands for the semi-simplification. 
As is well-known, 
$\psi_{i,j}$ satisfies the following properties 
(since the $G_K$-action on  $V_i$ is given by 
a character $G_K\to F_i^{\times}$):
\begin{itemize}
\item[(a)] $\psi_{i,1},\dots ,\psi_{i,d_{F_i}}$ are $\mbb{Q}_p$-conjugate with each other, that is,
$\psi_{i,k}=\tau_{k\ell}\circ \psi_{i,\ell}$ for some  $\tau_{k\ell}\in G_{\mbb{Q}_p}$, and 
\item[(b)] $\psi_{i,1},\dots ,\psi_{i,d_{F_i}}$ have values in a 
$p$-adic field $M_i$ (in the fixed algebraic closure $\overline{\mbb{Q}}_p$ of $\mbb{Q}_p$)
 which is  
$\mbb{Q}_p$-isomorphic\footnote{Note that $K$ lives in our fixed algebraic closure 
$\overline{\mbb{Q}}_p$ of $\mbb{Q}_p$ but 
 $F_i$ does not 
lives in $\overline{\mbb{Q}}_p$.} 
to the Galois closure of $F_i/\mbb{Q}_p$
(in an algebraic closure of $F_i$).
We remark that $d_{M_i}$ divides $d_{F_i}!$.
\end{itemize}
In particular, we have 
$$
v_p(\det \rho_i(\sigma)-E)=d_{F_i}v_p(\psi_i(\sigma)-1),
$$
where $\psi_i:=\psi_{i,1}$.
Let $M$ be the composite field of $M_1,\dots ,M_{n}$,
and we regard $\psi_1,\dots, \psi_{n}$
as characters of $G_K$ with values in $M^{\times}$;  $\psi_i\colon G_K\to M^{\times}$. 
The field $M$ is a Galois extension of $\mbb{Q}_p$ in $\overline{\mbb{Q}}_p$ and 
$d_M$ divides $d_{F_1}!d_{F_2}!\cdots d_{F_n}!$.
Since $\sum^n_{i=1}d_{F_i}=2g$, 
we find 
\begin{equation}
\label{Msize}
d_M\mid (2g)!.
\end{equation} 
(Here, we recall that the product of consecutive $n$ natural numbers is divided by $n!$
for any natural number $n$.)
In particular, we have $M\cap k=\mbb{Q}_p$
since $d_k$ is prime to $(2g)!$, and then we obtain
$$
\ker \mrm{Nr}_{M/\mbb{Q}_p}\subset \ker \mrm{Nr}_{Mk/k}\subset \ker \mrm{Nr}_{K_Mk/k}.
$$
Here, $K_M$ is the composite $KM$ of $K$ and $M$.
It follows from Proposition \ref{vplem} that   we obtain 
\begin{align}
\label{later}
c 
& \le  \mrm{Min}\left\{v_p(\det (\rho(\sigma)-E)))\mid \sigma\in {G_{K_Mk_{\pi}}}\right\} 
=  \mrm{Min}\left\{ \sum^n_{i=1} d_{F_i}v_p(\psi_i(\sigma)-1) \mid \sigma\in G_{K_Mk_{\pi}} \right\}  
\notag \\
& \le
\mrm{Min}\left\{ \sum^n_{i=1} d_{F_i}v_p(\psi_{i,K_Mk}(\pi \omega)^{-1}-1) 
\mid \omega\in \ker \mrm{Nr}_{K_Mk/k} \right\} \notag \\
& \le
\mrm{Min}\left\{ \sum^n_{i=1} d_{F_i}v_p(\psi_{i,K_Mk}(\pi \omega)^{-1}-1) 
\mid \omega\in \ker \mrm{Nr}_{M/\mbb{Q}_p} \right\} \notag \\
&  \le
\mrm{Min}\left\{ \sum^n_{i=1} d_{F_i}v_p(\psi^{\mu}_{i,K_Mk}(\pi \omega)^{-1}-1) 
\mid \omega\in \ker \mrm{Nr}_{M/\mbb{Q}_p} \right\}. 
\end{align}
Here, $\mu$ is the integer appeared in the statement 
of Theorem \ref{MT:CM:refined}.
Note that $\psi_i$ is a crystalline character since $A$ has good reduction over $K$.
By rearranging the numbering of subscripts,
we may suppose the following situation for some $0\le r\le n$.
\begin{itemize}
\item[(I)] For  $1\le i\le r$, the set of the Hodge-Tate weights of 
$M(\psi_i)$ is  $\{0,1\}$. 
\item[(II)] For  $r< i\le n$, the set of  the Hodge-Tate weights of $M(\psi_i)$ 
is either $\{1\}$ or $\{0\}$.
\end{itemize}

\begin{lemma}
\label{key:Wei}
%Let the notation be as above.
For  $r< i\le n$ and any $\omega\in \ker \mrm{Nr}_{M/\mbb{Q}_p}$,
we have 
$$
v_p(\psi^{\mu}_{i,K_Mk}(\pi\omega)^{-1}-1) \le L_g((2g)!\cdot d_{Kk/k}f_k\cdot \mu).
$$
\end{lemma}

\begin{proof}
In this proof we  set $L:=K_Mk$.
We know that  the morphism
$\psi_{i,\mrm{alg}}\colon \underline{L}^{\times}\to \underline{M}^{\times}$
corresponding to $\psi_i|_{G_{L}}$ is trivial or 
$\mrm{Nr}_{L/\mbb{Q}_p}^{-1}$ on $\mbb{Q}_p$-points. 
This in particular gives  $\psi_{i,L}(\omega)=1$.
Since $\pi_L^{e_{L/k}}\pi^{-1}$ is a $p$-adic unit
for any uniformizer $\pi_L$ of $L$, 
we  find
\begin{align*}
\psi_{i,L}(\pi\omega)^{-1}
&=\psi_{i,L}(\pi)^{-1}
=
\psi_{i,L}(\pi_L^{-e_{L/k}}\cdot \pi_L^{e_{L/k}}\pi^{-1}) 
\notag \\
& = 
\alpha_i^{-e_{L/k}}\cdot \psi_{i,\mrm{alg}}(\pi)^{-1}
\end{align*}
where $\alpha_i:=\psi_{i,L}(\pi_L)\psi_{i,\mrm{alg}}(\pi_L)^{-1}$. 
Denote by $L'$ the unramified extension of $L$ of degree $\mu e_{L/k}$.

(I) Suppose that the set of  the Hodge-Tate weights of $M(\psi_i)$ is $\{0\}$.
In this case $\psi_{i,\mrm{alg}}$ is  trivial and thus we have 
$\psi^{\mu}_{i,L}(\pi \omega)^{-1}=\alpha_i^{-\mu e_{L/k}}$.
It follows from Lemma 9 of \cite{Oz1} that 
$\psi^{\mu}_{i,L}(\pi \omega)^{-1}$ 
is a unit root of 
the characteristic polynomial  $f(T)$ of 
the geometric Frobenius endomorphism of  
$\overline{A}_{/\mbb{F}_{L'}}$.
Since $f(1)=\sharp \overline{A}(\mbb{F}_{q_{L'}})$,
we see 
$v_p(\psi^{\mu}_{i,L}(\pi \omega)^{-1}-1)
\le v_p(\sharp \overline{A}(\mbb{F}_{q_{L'}}))
\le \left[\log_p\sharp \overline{A}(\mbb{F}_{q_{L'}}) \right]$. 
It follows from the Weil bound that 
$
v_p(\psi^{\mu}_{i,L}(\pi \omega)^{-1}-1)\le 
L_g(f_{L'}). 
$
Since 
we have 
$
f_{L'}  = \mu e_{L/k}f_L
 = d_{L/Kk} \cdot \mu  \cdot d_{Kk/k} f_{k} 
\le  (2g)! \cdot \mu \cdot d_{Kk/k}f_k.
$
we obtain the desired inequality.

(II) Suppose that the set of  the Hodge-Tate weights of $M(\psi_i)$ is $\{1\}$.
In this case $\psi_{i,\mrm{alg}}$ is  $\mrm{Nr}_{L/\mbb{Q}_p}^{-1}$ on $\mbb{Q}_p$-points.
If  we set $\beta:=q_k^{-1}\mrm{Nr}_{k/\mbb{Q}_p}(\pi)$,  
we find  
\begin{align*}
\psi^{\mu}_{i,L}(\pi \omega)^{-1}-1
& = (\alpha_i^{-1}\mrm{Nr}_{k/\mbb{Q}_p}(\pi)^{f_{L/k}})^{\mu e_{L/k}} -1 \notag \\
& = ((\alpha_i^{-1}q_L)^{\mu e_{L/k}}-1)
\beta^{\mu d_{L/k}} + (\beta^{\mu d_{L/k}}-1).
\end{align*} 
It again follows from Lemma 9 of \cite{Oz1}   
that 
$(\alpha_i^{-1}q_L)^{\mu e_{L/k}}$ 
is a unit root of 
the characteristic polynomial  $f^{\vee}(T)$ of 
the geometric Frobenius endomorphism of  
$\overline{A^{\vee}}_{/\mbb{F}_{L'}}$.
Since $f^{\vee}(1)=\sharp \overline{A^{\vee}}(\mbb{F}_{q_{L'}})$,
the same argument as in (I) shows that  
$v_p((\alpha_i^{-1}q_L)^{\mu e_{L/k}}-1)
\le L_g(f_{L'})\le L_g((2g)!\cdot \mu \cdot d_{Kk/k}f_k)$.  
In particular, we have  $v_p(\beta^{\mu d_{L/k}}-1)>v_p((\alpha_i^{-1}q_L)^{\mu e_{L/k}}-1)$
by the assumption  (i)'.
Since $\beta$ is a $p$-adic unit, 
we obtain 
$v_p(\psi^{\mu}_{i,L}(\pi\omega)^{-1}-1) 
= v_p((\alpha_i^{-1}q_L)^{\mu e_{L/k}}-1) 
\le L_g((2g)!\cdot \mu \cdot d_{Kk/k}f_k)$
as desired. 
\end{proof}
By \eqref{later} and the lemma, in the case where $r=0$,
we have 
\begin{equation}
\label{r=0}
c\le \sum^n_{i=1} d_{F_i}L_g((2g)!\cdot \mu \cdot d_{Kk/k}f_k)
=2gL_g((2g)!\cdot \mu \cdot d_{Kk/k}f_k). 
\end{equation}
In the rest of the proof, we assume that $r>0$.
By \eqref{later}  and the lemma again, we have 
\begin{align*}
c 
&  \le
\mrm{Min}\left\{ \sum^r_{i=1} d_{F_i}v_p(\psi^{\mu}_{i,K_Mk}(\pi \omega)^{-1}-1) 
\mid \omega\in \ker \mrm{Nr}_{M/\mbb{Q}_p} \right\} \\
& \quad +L_g((2g)!\cdot \mu \cdot d_{Kk/k}f_k)\sum^n_{i=r+1} d_{F_i}. 
\end{align*}
Here we remark that    $v_p(\mu)=0$ and
the Hodge-Tate weights of $\psi_i^{\mu}$ for each $1\le i\le r$ consist of 
$0$ and $\mu$. 
Hence, applying Theorem \ref{key:est2} to the set of characters 
$\psi_1^{\mu},\dots , \psi_r^{\mu}\colon G_{K_Mk}\to M^{\times}$, an element $x=\pi$ and $h=0$, 
there exists an element $\hat{\omega}\in \ker \mrm{Nr}_{M/\mbb{Q}_p}$ 
and an integer $0\le s=s(\pi)\le r$ as in the theorem. Then we obtain
\begin{align*}
c 
& \le \sum^r_{i=1} d_{F_i}v_p(\psi^{\mu}_{i,K_Mk}(\pi \hat{\omega}^{p^s})^{-1}-1) 
+L_g((2g)!\cdot \mu \cdot d_{Kk/k}f_k)\sum^n_{i=r+1} d_{F_i} \\
& \le \sum^r_{i=1} d_{F_i}(r+\delta_{(i)}+C(d_{K_Mk},M,0)) 
+L_g((2g)!\cdot \mu \cdot d_{Kk/k}f_k)\sum^n_{i=r+1} d_{F_i}\\
& \le 2g\Delta_0 +\sum^r_{i=1} d_{F_i}(r+\delta_{(i)})
\end{align*}
where $\Delta_0:=\mrm{Max}\left\{ C(d_{K_Mk},M,0), 
L_g((2g)!\cdot \mu \cdot d_{Kk/k}f_k) \right\}$.
Since $d_M$ divides $(2g)!$, 
we also have 
$$
C(d_{K_Mk},M,0)
<v_p(d_{Kk})+ \frac{(2g)!}{2} \left((2g)!+v_p((2g)!) + v_p(2)((2g)!-1)\right).
$$
Thus, for the constant
$\Delta_g(K,k)$
defined in the statement of the proposition,
we obtain  $\Delta_0\le \Delta_g(K,k)$ and 
$
c \le 2g\Delta_g(K,k) +\sum^r_{i=1} d_{F_i}(r+\delta_{(i)}).
$
\begin{itemize}
\item If $r\le 2$, we have 
$\sum^r_{i=1} d_{F_i}(r+\delta_{(i)})=
\sum^r_{i=1} d_{F_i}r\le r\cdot 2g\le 4g$.
\item If $r>2$, we have 
$\sum^r_{i=1} d_{F_i}(r+\delta_{(i)})
=r\sum^r_{i=1} d_{F_i} 
+\sum^r_{i=3} d_{F_i}\delta_{(i)}
\le n\sum^n_{i=1} d_{F_i} + \sum^n_{i=3} d_{F_i}(2n-5)
\le n\cdot 2g +(2n-5) (\sum^n_{i=1} d_{F_i}-2)
\le 2g\cdot 2g +(4g-5) \cdot (2g-2)
=12g^2-18g+10$.
\end{itemize}
Therefore, for any $r>0$,  we find 
$$
c \le 2g\Delta_g(K,k) +12g^2-18g+10.
$$
Note that this inequality holds also for the case $r=0$ by \eqref{r=0}.
Now the proposition follows from \eqref{order:det}.
\end{proof}

%%%%%%%%%%%%%%%%%%%%%%%%%%%%%%%%%%%%%%%%%%%%%%%%%%%%%%%%%
%%%%%%%%%%%%%%%%%%%%%%%%%%%%%%%%%%%%%%%%%%%%%%%%%%%%%%%%%%
%%%%%%%%%%%%%%%%%%%%%%%%%%%%%%%%%%%%%%%%%%%%%%%%%%%%%%%%%%

\subsection{General cases}
We show Theorems \ref{MT:CM:refined}
and \ref{MT:CM:KT}.
For this, we need the following observations given by Serre-Tate \cite{ST} and Silverberg \cite{Si1}.
\begin{theorem}
\label{ST-Sil}
Let $A$ be a $g$-dimensional abelian variety over $K$.

\noindent
{\rm (1)}  Put $m=3$ or $m=4$ if $p\not=3$ or $p=3$, respectively.
Then  $A$ has semi-stable reduction over $K(A[m])$ 
and all the endomorphisms of $A$
are defined over $K(A[m])$.

\noindent
{\rm (2)}  Let $L$ be the intersection 
of the fields $K(A[N])$ for all integers $N>2$.
 Then, all the endomorphisms of $A$
are defined over $L$ and $[L:K]$ divides $H(g)$.

\noindent
{\rm (3)} Assume that $A$ has potential good reduction. 
Let $\rho_{A,\ell}\colon G_K\to GL_{\mbb{Z}_p}(T_{\ell}(A))$
be the continuous homomorphism defined by the $G_K$-action on 
the Tate module $T_{\ell}(A)$ 
for any prime $\ell$.
\begin{itemize}
\item[{\rm (3-1)}] For any  prime $\ell$ not equal to $p$, 
let $H_{\ell}$ be  the kernel of the restriction of 
$\rho_{A,\ell}$ to $I_K$.
Then $H_{\ell}$ is an open subgroup of $I_K$, which is 
independent of the choice of $\ell$. 
Moreover, if we set $c:=[I_K:H_{\ell}]$, then 
there exists a  finite totally ramified extension $L/K$ of degree  $c$
such that 
$A$ has good reduction over $L$. 
\item[{\rm (3-2)}] If  $A$ has complex multiplication and 
all the endomorphisms of $A$
are defined over $K$, then the constant 
$c$ above satisfies $c\le \Phi(g)$.
\end{itemize}

\noindent
{\rm (4)} Assume that $A$ has complex multiplication {\rm (}over $\overline{K}${\rm )}.
Then, there exists a finite extension $L/K$ of degree at most $\Phi(g)H(g)$ 
such that 
$A$ has good reduction over $L$ 
and all the endomorphisms of $A$
are defined over $L$.
\end{theorem}

\begin{proof}
(1) follows from  \cite[Theorem 4.1]{Si1} 
and the Raynaud's criterion of semi-stable reduction \cite[Proposition 4.7]{Gr}.
(2) is \cite[Theorem 4.1]{Si1}, and (4) is an immediate 
consequence  of (2) and (3) since $A$ must have
potential good reduction 
under the condition that $A$ has complex multiplication.
The first statement related to  $H_{\ell}$ in (3-1) is \cite[\S 2, Theorem 2, p.496]{ST}.
The rest assertions of (3) are also essentially consequences of 
results given in 
\S 2 and \S 4  of  \cite{ST} 
but it is not directly mentioned in {\it loc., cit}.  
Thus we give a proof here,  just in case.  
The group $H$ is a closed normal subgroup of $G_K$, which is 
also open in $I_K$. 
Let $\Gamma$ be the closure of the subgroup of $G_K$  
generated by any choice of a lift of the $q_K$-th Frobenius element
in $G_{\mbb{F}_{q_K}}$.
The projection $G_K\to G_{\mbb{F}_{q_K}}$ gives an isomorphism of $\Gamma$
onto $G_{\mbb{F}_{q_K}}$; in particular, 
$G_K$ is the semi-direct product of $\Gamma$ and $I_K$.
Let $K_{\Gamma}/K$ be the field extension (of infinite degree) corresponding to 
$\Gamma\subset G_K$, and let
$M/K^{\mrm{ur}}$ be the finite extension corresponding to 
$H:=H_{\ell}\subset I_K$. 
Note that $A$ has good reduction over $M$.
Now we set $L:=K_{\Gamma} \cap M$.
Then $L/K$ is  totally ramified since so is $K_{\Gamma}/K$.
Furthermore, it is immediate to check $H\Gamma\cap I_K=H$;
this shows $LK^{\mrm{ur}}=M$.
Hence we obtain that $A$ has good reduction over $L$
and $[L:K]=[M:K^{\mrm{ur}}]=c$.
This shows (3-1).
Next we show (3-2).  
Let $F$ be the number field of degree $2g$ of complex multiplication of $A$.
Then $V_{\ell}(A)$ has a structure of free $F\otimes_{\mbb{Q}} \mbb{Q}_{\ell}$
module of rank one and the $G_K$-action on $V_{\ell}(A)$ commutes with 
$F\otimes_{\mbb{Q}} \mbb{Q}_{\ell}$.
Thus we may consider $\rho_{A,\ell}$ as a character 
$G_K\to (F\otimes_{\mbb{Q}} \mbb{Q}_{\ell})^{\times}$.
Moreover, the image of this character restricted to $I_K$ has values in 
the group $\mu(F)$ of roots of unity contained  in $F$
by \cite[\S 4, Theorem 6, p.503]{ST}. 
Thus we obtain the fact that  $c$ divides the order $m$ of $\mu(F)$. 
On the other hand,  since $\mu_m$ is a subset of $F$,
we have $\varphi(m)\mid 2g$.
Therefore, we obtain $c\le m\le \Phi(g)$ as desired.
\end{proof}

Now we are ready to show our main theorems. 
First we show Theorem \ref{MT:CM:refined}.
\begin{proof}[Proof of Theorem \ref{MT:CM:refined}]
Let $A$ be as in the theorem. 
Since $A$ has complex multiplication, 
it follows from Theorem \ref{ST-Sil} (4) 
that there exists a finite extension $L/K$ 
such that 
$d_{L/K}\le \Phi(g)H(g)$, 
$A$ has good reduction over $L$ 
and all the endomorphisms of $A$
are defined over $L$.
In addition, we have 
$v_p((q_k^{-1}\mrm{Nr}_{k/\mbb{Q}_p}(\pi))^{\mu}-1)
>g\cdot (2g)!\cdot \Phi(g)H(g) \cdot \mu \cdot d_{Kk/k}f_k
=L_g((2g)!\cdot \Phi(g)H(g) \cdot \mu \cdot d_{Kk/k}f_k)
\ge L_g((2g)!\cdot   \mu \cdot d_{Lk/k}f_k)$
by the assumption (i)  and Remark \ref{Lg} (2).
Thus we can apply  Proposition \ref{MT:CM:good} to $A/L$; 
we have 
$$
A(Lk_{\pi})[p^{\infty}]) \subset A[p^{C'}]
$$
where $C'=2g\Delta_g(L,k) + 12g^2-18g+10$.
Here, 
\begin{align*}
C_g(L,k) & =v_p(d_{Lk})+ \frac{(2g)!}{2} \left((2g)!+v_p((2g)!) + v_p(2)((2g)!-1)\right), \\
\Delta_g(L,k)& =\mrm{Max}\left\{ C_g(L,k), L_g((2g)!\cdot   \mu \cdot d_{Lk/k}f_k) \right\}.
\end{align*}
Note that we have 
$v_p(d_{Lk})< d_{Lk}\le \Phi(g)H(g)\cdot d_{Kk}$ 
and 
$L_g((2g)!\cdot   \mu \cdot d_{Lk/k}f_k)
\le  g\cdot (2g)!\cdot \Phi(g)H(g)\cdot \mu \cdot d_{Kk}$.
Therefore, it suffices to show
$$
\Phi(g)H(g)\cdot  d_{Kk} + \frac{(2g)!}{2} \left((2g)!+v_p((2g)!) + v_p(2)((2g)!-1)\right)
<g\cdot (2g)!\cdot \Phi(g)H(g)\cdot \mu \cdot d_{Kk}
$$
for the proof but this is clear.
\end{proof}

\begin{remark}
\label{MT:rem}
In the above proof of Theorem \ref{MT:CM:refined}, 
we referred the field extension $L/K$  of Theorem  \ref{ST-Sil} (4) and 
the upper bound $\Phi(g)H(g)$ of $[L:K]$. 
By Theorem \ref{ST-Sil} (1),  
we may refer the field $K(A[m])$ instead of the above $L$.
Since we have a natural embedding from $\mrm{Gal}(K(A[m])/K)$
into $GL(A[m])\simeq GL_{2g}(\mbb{Z}/m\mbb{Z})$,
we obtain a bound for the extension degree of $K(A[m])/K$; 
we have 
$[K(A[m])/K]\le G(g)$, where 
\begin{align*}
G(n):=\sharp GL_{2n}(\mbb{Z}/m\mbb{Z})
=
\left\{
\begin{array}{cl}
\prod^{2n-1}_{i=0}(3^{2n}-3^i) &\quad 
\mrm{if}\ p\not=3, \cr
2^{4n^2}\prod^{2n-1}_{i=0}(2^{2n}-2^i) 
&\quad \mrm{if}\ p=3.
\end{array}
\right.
\end{align*}
for $n>0$. 
Note that we have $G(n)<m^{4n^2}$.
It is not difficult to check  the inequalities  $\Phi(1)H(1) > G(1)$ and 
$\Phi(g)H(g) < G(g)$ for   $g>1$ (see Section \ref{Phi-H} below).
Hence, only in the case $g=1$ of elliptic curves, 
we can obtain smaller bound than that given in Theorem \ref{MT:CM:refined}
by replacing $\Phi(g)H(g)$ with $G(1)$.
\end{remark}

Applying Theorem \ref{MT:CM} with $k=\mbb{Q}_p$ and $\pi=p$,
we  immediately obtain the following.
\begin{corollary}
\label{MT:CM:cycl}
Let $A$ be a $g$-dimensional abelian variety  over a $p$-adic field 
$K$ with complex multiplication.
Then we have 
$$
A(K(\mu_{p^{\infty}}))[p^{\infty}] \subset A[p^{C}]
$$
where
\begin{align*}
C & := 2g^2\cdot (2g)!\cdot \Phi(g)H(g)\cdot d_{K} 
+12g^2-18g+10 
\end{align*}
In particular, we have 
$$
\sharp A(K(\mu_{p^{\infty}}))[p^{\infty}]\le p^{2gC}.
$$
\end{corollary}
Next we show Theorem \ref{MT:CM:KT}.

\begin{proof}[Proof of Theorem \ref{MT:CM:KT}]
We follow essentially the same argument as that of Theorem \ref{MT:CM:refined}.
Put $\hat{K}=K(\sqrt[p^{\infty}]{K})$.

{\bf Step 1.} 
First we consider  the case  where $A$ has good reduction over $K$ 
and all the endomorphisms of $A$ are defined over $K$.
Put $\nu=v_p(d_K)+1+v_p(2)$ and 
\begin{align*}
C_g(K) & =v_p(d_K) +\nu+ \frac{(2g)!}{2} \left((2g)!+v_p((2g)!) + v_p(2)((2g)!-1)\right), \\
\Delta_g(K) & =\mrm{Max}\left\{ C_g(K), L_g((2g)!\cdot p^{\nu} \cdot d_K) \right\}.
\end{align*}
Following the proof of  Proposition \ref{MT:CM:good}, 
we show 
\begin{equation}
\label{hatK}
A(\hat{K})[p^{\infty}] \subset A[p^{C'}]
\end{equation}
where 
$
C':= 2g\Delta_g(K) + 12g^2-18g+10.
$
Let $\rho\colon G_K\to GL_{\mbb{Z}_p}(T_p(A))\simeq GL_{2g}(\mbb{Z}_p)$,
$M/\mbb{Q}_p$ 
and $\psi_1,\dots,\psi_n\colon G_K\to M^{\times}$
be as in the proof of  Proposition \ref{MT:CM:good}.
If we denote by $\hat{K}_{\mrm{ab}}$ 
the maximal abelian extension of $K$ contained in $\hat{K}$, 
all the points of $A(\hat{K})[p^{\infty}]$ are in fact defined over 
$\hat{K}_{\mrm{ab}}$  since $\rho$ is abelian.
Thus, setting 
$
c:=\mrm{Min}\{v_p(\det (\rho(\sigma)-E)))\mid \sigma\in {G_{\hat{K}_{\mrm{ab}}}} \}, 
$
we find 
\begin{equation}
\label{order:det2}
A(\hat{K})[p^{\infty}] = A(\hat{K}_{\mrm{ab}})[p^{\infty}]   \subset A[p^c]
\end{equation}
if $c$ is finite (see arguments just above \eqref{order:det}).
On the other hand, 
we set $G:=\mrm{Gal}(\hat{K}/K)$ and $H:=\mrm{Gal}(\hat{K}/K(\mu_{p^{\infty}}))$. 
Let $\chi_p\colon G_K\to \mbb{Z}_p^{\times}$ be the $p$-adic cyclotomic character.
Since we have $\sigma \tau \sigma^{-1}=\tau^{\chi_p(\sigma)}$ for any $\sigma\in G$
and $\tau\in H$,
we see $(G,G)\supset (G,H)\supset H^{\chi_p(\sigma)-1}$.
Hence we have a natural surjection
\begin{equation}
\label{H}
H/H^{\chi_p(\sigma)-1}
\twoheadrightarrow 
H/\overline{(G,G)}
=\mrm{Gal}(\hat{K}_{\mrm{ab}}/K(\mu_{p^{\infty}}))
\quad \mbox{for any $\sigma\in G$}.
\end{equation}
\begin{lemma}
We have $\chi_p(\sigma_0)-1=p^{\nu}$ 
for some $\sigma_0\in G$.
\end{lemma}
\begin{proof}
We denote by $K'$ the field $K(\mu_p)$ or $K(\mu_4)$ if $p\not=2$ or $p=2$,
respectively. If we denote by $p^{\ell}$ the 
order of the set of $p$-power roots of unity in $K'$,
we see $K'\cap \mbb{Q}_p(\mu_{p^{\infty}})=\mbb{Q}_p(\mu_{p^{\ell}})$ and thus
$\chi_p(G_{K'})=1+p^{\ell}\mbb{Z}_p$. 
Furthermore, since  
$[\mbb{Q}_p(\mu_{p^{\ell}}):\mbb{Q}_p]$ divides $[K':K][K:\mbb{Q}_p]$,
we see $p^{\ell-1-v_p(2)}\mid d_K$. 
Hence we obtain $\chi_p(G_{K'})\supset 1+p^{\nu}\mbb{Z}_p$ and the lemma follows.
\end{proof}
By the lemma above and \eqref{H}, we see that  
$\mrm{Gal}(\hat{K}_{\mrm{ab}}/K(\mu_{p^{\infty}}))$
is of exponent $p^{\nu}$, that is, 
$\sigma\in G_{K(\mu_{p^{\infty}})}$ implies 
$\sigma^{p^{\nu}}\in G_{\hat{K}_{\mrm{ab}}}$.
This shows 
$
c\le \mrm{Min}\{v_p(\det (\rho(\sigma)^{p^{\nu}}-E)))\mid 
\sigma\in G_{K(\mu_{p^{\infty}})} \}.
$
Mimicking the arguments for inequalities  \eqref{later}, 
we find 
$$
c  \le
\mrm{Min}\left\{ \sum^n_{i=1} d_{F_i}v_p(\psi^{p^{\nu}}_{i,K_M}(\pi \omega)^{-1}-1) 
\mid \omega\in \ker \mrm{Nr}_{M/\mbb{Q}_p} \right\}.
$$
Now the inequality \eqref{order:det2} follows 
by completely the same method  as the proof of Proposition 
\ref{MT:CM:good} (with replacing the pair $(k,\mu)$ there with $(\mbb{Q}_p,p^{\nu})$).

{\bf Step 2.} 
Next we consider the general case.
Since $A$ has complex multiplication, 
it follows from Theorem \ref{ST-Sil} (4) 
that there exists a finite extension $L/K$ 
such that 
$d_{L/K}\le \Phi(g)H(g)$, 
$A$ has good reduction over $L$ 
and all the endomorphisms of $A$
are defined over $L$. 
Thus we can apply  the result of Step 1 to $A/L$; 
we have 
$$
A(\hat{K})[p^{\infty}] \subset A(\hat{L})[p^{\infty}] \subset A[p^{C''}]
$$
where $C'':=2g\Delta_g(L) + 12g^2-18g+10$.
We find  
\begin{align*}
L_g((2g)!\cdot p^{v_p(d_L)+1+v_p(2)}\cdot d_L)
&= L_g((2g)!\cdot p^{1+v_p(2)}\cdot 
p^{v_p(d_{L/K})}d_{L/K} \cdot p^{v_p(d_K)}d_K) \\
& \le L_g((2g)!\cdot p^{1+v_p(2)}\cdot 
(d_{L/K})^2 \cdot p^{v_p(d_K)}d_K) \\
& \le g\cdot (2g)!\cdot p^{1+v_p(2)}\cdot 
(\Phi(g)H(g))^2 \cdot p^{v_p(d_K)}d_K.
\end{align*}
(For the last equality, see Remark \ref{Lg} (2).)
Now Theorem \ref{MT:CM:KT} immediately follows by 
 $\Delta_g(L)\le g\cdot (2g)!\cdot p^{1+v_p(2)}\cdot 
(\Phi(g)H(g))^2 \cdot p^{v_p(d_K)}d_K$.
\end{proof}

%%%%%%%%%%%%%%%%%%%%%%%%%%%%%%%%%%%%%%%%%%%%%%%%%%%%%%%%%%%%%%%%%
%%%%%%%%%%%%%%%%%%%%%%%%%%%%%%%%%%%%%%%%%%%%%%%%%%%%%%%%%%%%%%%%%
%ordinary
%%%%%%%%%%%%%%%%%%%%%%%%%%%%%%%%%%%%%%%%%%%%%%%%%%%%%%%%%%%%%%%%%
%%%%%%%%%%%%%%%%%%%%%%%%%%%%%%%%%%%%%%%%%%%%%%%%%%%%%%%%%%%%%%%%%

One of the keys for our arguments above  
is a theory of locally algebraic representations.
Thus our method  essentially works also for abelian varieties $A$
with the property that the $G_K$-action on the semi-simplification of 
$V_p(A)\otimes_{\mbb{Q}_p} \overline{\mbb{Q}}_p$
is abelian.
For example, this is the case where  $A$  has good ordinary reduction.

\begin{proposition}
\label{val:ord}
Let $g>0$ be a positive integer.
Let $K$ and $k$ be $p$-adic fields.
Let $\pi$ be a uniformizer of $k$.
Assume that  
$q_k^{-1}\mrm{Nr}_{k/\mbb{Q}_p}(\pi)$ is a root of unity;
we denote by $0<\mu<p$ the minimum integer so that 
$(q_k^{-1}\mrm{Nr}_{k/\mbb{Q}_p}(\pi))^{\mu}=1$.
Then, for any $g$-dimensional abelian variety $A$  over  $K$
with good ordinary reduction, 
we have 
$$
A(Kk_{\pi})[p^{\infty}] \subset A[p^{2gL_g(\mu d_{Kk/k}f_k)}].
$$
In particular, we have 
$$
\sharp A(Kk_{\pi})[p^{\infty}]
 \le p^{4g^2L_g(\mu d_{Kk/k}f_k)}  < p^{4g^3(\mu d_{Kk/k}f_k+1+v_p(2))}.
$$
\end{proposition}
\begin{proof}
Put $V=V_p(A)$, $T=T_p(A)$ and 
$
c=\mrm{Min}\{v_p(\det (\rho(\sigma)-E)))\mid \sigma\in {G_{Kk_{\pi}}}\}.
$
By the same argument as the beginning of the proof 
of Proposition \ref{MT:CM:good},
we obtain 
\begin{equation}
\label{order:det:ord}
A(Kk_{\pi})[p^{\infty}] \subset A[p^c]
\end{equation}
if $c$ is finite.
Since $A$ has good ordinary reduction, we have an exact sequence 
$0\to V_1\to V\to V_2\to 0$ of $\mbb{Q}_p[G_K]$-modules
with the following properties.
\begin{itemize} 
\item[(i)] $V_1\simeq W\otimes_{\mbb{Q}_p} \mbb{Q}_p(1)$ for some unramified representation $W$ of $G_K$, and 
\item[(ii)] $ V_2$ is unramified. 
\end{itemize} 
Hence, taking a $p$-adic field $M$ large enough,  
we have $(V\otimes_{\mbb{Q}_p} M)^{\mrm{ss}}\simeq 
\oplus^{2g}_{i=1} M(\psi_i)$
for some continuous crystalline characters 
$\psi_i\colon G_K\to M^{\times}$.
Furthermore, 
for every $i$, the set of  the Hodge-Tate weights of $M(\psi_i)$ 
is either $\{1\}$ or $\{0\}$.
By Proposition \ref{vplem},  we have 
$c\le \sum^{2g}_{i=1}v_p(\psi^{\mu}_{i,Kk}(\pi)^{-1}-1)$.
Let $K'$ be the unramified extension of $Kk$
of degree $\mu e_{Kk/k}$.
By a similar method of the proof of Lemma \ref{key:Wei}, 
we find that 
$\psi^{\mu}_{i,Kk}(\pi)^{-1}$ is 
a unit root of 
the characteristic polynomial  $f(T)$ of 
the geometric Frobenius endomorphism of  
$\overline{A}_{/\mbb{F}_{K'}}$,  
otherwise 
$\psi^{\mu}_{i,Kk}(\pi)^{-1}$ is  a unit root of 
the characteristic polynomial  $f^{\vee}(T)$ of 
the geometric Frobenius endomorphism of  
$\overline{A^{\vee}}_{/\mbb{F}_{K'}}$.
We know  $f(1)=\sharp \overline{A}(\mbb{F}_{q_{K'}})$ and
$f^{\vee}(1)=\sharp \overline{A^{\vee}}(\mbb{F}_{q_{K'}})$,
and their $p$-adic valuations  are bounded by 
$L_g(f_{K'})$ by the Weil bound.
Since 
we have 
$
f_{K'}  = f_{K'/Kk}f_{Kk}=\mu d_{Kk/k}f_k, 
$
we obtain 
$c\le \sum^{2g}_{i=1}v_p(\psi^{\mu}_{i,Kk}(\pi)^{-1}-1)\le 2g L_g(\mu d_{Kk/k}f_k)$.
Now the result follows from \eqref{order:det:ord}.
\end{proof}
%%%%%%%%%%%%%%%%%%%%%%%%%%%%%%%%%%%%%%%%%%%%%%%%%%%%%%%%%%%%%%%%%
%%%%%%%%%%%%%%%%%%%%%%%%%%%%%%%%%%%%%%%%%%%%%%%%%%%%%%%%%%%%%%%%%
%ordinary
%%%%%%%%%%%%%%%%%%%%%%%%%%%%%%%%%%%%%%%%%%%%%%%%%%%%%%%%%%%%%%%%%
%%%%%%%%%%%%%%%%%%%%%%%%%%%%%%%%%%%%%%%%%%%%%%%%%%%%%%%%%%%%%%%%%

%%%%%%%%%%%%%%%%%%%%%%%%%%%%%%%%%%%%%%%%%%%%%%%%%
%%%%%%%%%%%%%%%%%%%%%%%%%%%%%%%%%%%%%%%%%%%%%%%%%
%%%%%%%%%%%%%%%%%%%%%%%%%%%%%%%%%%%%%%%%%%%%%%%%%
%%%%%%%%%%%%%%%%%%%%%%%%%%%%%%%%%%%%%%%%%%%%%%%%%
%%%%%%%%%%%%%%%%%%%%%%%%%%%%%%%%%%%%%%%%%%%%%%%%%
%%%%%%%%%%%%%%%%%%%%%%%%%%%%%%%%%%%%%%%%%%%%%%%%%

\section{Abelian varieties  over number fields}

In this section, we suppose that $K$ is a number field.
The goal of this section is to give a proof of Theorem \ref{gsurf2} in Introduction.
The theorem is an immediate consequence of the following proposition. 
\begin{proposition}
Let $g,K,d$ and $h$ be as in Theorem \ref{gsurf2}.

\noindent
{\rm (1)} 
Let $A$ be a $g$-dimensional abelian variety over $K$ with semi-stable 
reduction everywhere.
Let $p_0$ be the smallest prime number such that 
$A$ has  good reduction at some finite place of $K$ above $p_0$.
Then $A(K(\mu_{\infty}))[p]$ is zero if 
$p>(1+\sqrt{p_0}^{dh})^{2g}$, $p$ is unramified in $K$
and $A$ has good reduction at some finite place of $K$ above $p$. 

\noindent
{\rm (2)} 
Let $A$ be a $g$-dimensional abelian variety over $K$ with complex multiplication
which has good reduction everywhere. 
Then, for any prime $p$, we have 
$$
 A(K(\mu_{\infty}))[p^{\infty}]\subset A[p^C]
$$
where $C:=2g^2\cdot (2g)!\cdot \Phi(g)H(g)\cdot dh
+12g^2-18g+10$.
\end{proposition}

\begin{proof}
Let $A$ be a $g$-dimensional abelian variety over $K$ with semi-stable reduction everywhere.
Let $K'$ be the maximal extension of  $K$ contained in $K(\mu_{\infty})$
which is unramified at all finite places of $K$. 
Note that $K'$ is a finite abelian extension of $K$.
In particular, it follows from class field theory that $[K':K]$ 
is a divisor of the narrow class number $h$ of $K$.
If we denote by $L_p$ the maximal extension of  
$K$ contained in $K(\mu_{\infty})$
which is unramified at all places except for places dividing $p$ and the infinite places,
then it is shown in \cite[Appendix, Lemma]{KL} that 
$L_p=K'(\mu_{p^{\infty}})$. 

(1) We give a proof of the assertion (1). 
Here we mainly follow Ribet's arguments in \cite{KL}.
We suppose that $p$ is prime to $2p_0$ 
and  also suppose that $p$ is  unramified in $K$.
Assume that $A(K(\mu_{\infty}))[p]\not= O$. 
We claim that there exists a $g$-dimensional abelian variety $A'$ over $K'$
which is $K'$-isogenous to $A$ such that 
 $A'(K')[p]\not= O$. 
We denote by $G$ and $H$ the absolute Galois groups of $K'$ and $K(\mu_{\infty})$,
respectively.
The assumption $A(K(\mu_{\infty}))[p]\not=O$  
is equivalent to say that $A[p]^H\not =O$. 
Let $W$ be a simple $G$-submodule of $A[p]^H$. 
Ribet showed in the proof of Theorem 2 of \cite{KL} that, 
since $A$ has semi-stable reduction everywhere over $K'$,
$W$ is one-dimensional over $\mbb{F}_p$ and  
the action of $G$ on $W$ factors through $\mrm{Gal}(K'(\mu_p)/K')$.
Since $p$ is unramified at $K'$,
we find that the $G$-action on $W$ is given by 
$\overline{\chi}_p^n$ for some $0\le n\le p-1$, where $\overline{\chi}_p$ 
is the mod $p$ cyclotomic character.
Moreover, since $A$ has good reduction at
some finite place of $K'$ above $p\ (\not=2)$, 
it follows from the classification of Tate and  Oort 
that $n$ is equal to $0$ or $1$.
Thus $W$ is isomorphic to $\mbb{F}_p$ or $\mbb{F}_p(1)$. 
If we are in the former case, we have $A'(K')[p]\not= O$ for $A':=A$.
Suppose that we are in the latter case.
Then there exists a surjection $A^{\vee}[p]\to \mbb{F}_p$ of $G$-modules.
If we denote by $C$ the kernel of this surjection, then 
the $G$-action on $A^{\vee}[p]$ preserves $C$.
This implies that  
$A':=A^{\vee}/C$ is an abelian variety defined over $K'$
and we find that there exists a trivial $G$-submodule of $A'[p]$ of order $p$.
Thus we have $A'(K')[p]\not=O$.  
This finishes  the proof of the claim.

Now we take a prime $\mfrak{p}'_0$ of $K'$ above $p_0$
such that $A$ has good reduction at $\mfrak{p}'_0$.
Since $A'$ above is $K'$-isogenous to $A$, 
we know that $A'$ has good reduction at  $\mfrak{p}'_0$ by \cite[\S 1, Corollary 2]{ST}.
If we denote by $K'_{\mfrak{p}'_0}$ the completion of $K'$ at $\mfrak{p}'_0$
and also denote by $\mbb{F}_{\mfrak{p}'_0}$ the residue field of $K'_{\mfrak{p}'_0}$, 
then reduction modulo $\mfrak{p}'_0$ gives an injective homomorphism
$$
A'(K')[p]\subset A'(K'_{\mfrak{p}'_0})[p]\hookrightarrow \bar{A}'(\mbb{F}_{\mfrak{p}'_0}).
$$
We recall that  $A'(K')[p]\not= O$. 
Since the order of  $\mbb{F}_{\mfrak{p}'_0}$ is bounded by 
$p_0^{dh}$, it follows from the Weil bound that 
we have 
$p<(1+\sqrt{p_0}^{dh})^{2g}$. 
This finishes the proof.

(2) We give a proof of the assertion (2). 
Let $A$ be an abelian variety as in the statement.
Since $A$  has good reduction everywhere over $K$,
it follows from the criterion of N\'eron-Ogg-Shafarevich 
that the $G_K$-action on $A[p^{\infty}]$ is unramified outside $p$.
This gives the fact that 
the $G_K$-action on $A(K(\mu_{p^{\infty}}))[p^{\infty}]$ factors through 
$\mrm{Gal}(L_p/K)=\mrm{Gal}(K'(\mu_{p^{\infty}})/K)$.
Thus we have 
$$
A(K(\mu_{\infty}))[p^{\infty}]
=A(K'(\mu_{p^{\infty}}))[p^{\infty}].
$$
Since we have $[K':\mbb{Q}]\le dh$, the result follows from 
Corollary \ref{MT:CM:cycl}.
\end{proof}

\section{Bounds on $\Phi(n)$ and $H(n)$}
\label{Phi-H}
We recall the definitions of $\Phi(n)$ and $H(n)$:
\begin{align*}
& \Phi(n):=
\mrm{Max}\{m\in \mbb{Z} \mid \mbox{$\varphi(m)$ divides $2n$} \}, \\
& H(n):=
\mrm{gcd}\{\sharp \mrm{GSp}_{2n}(\mbb{Z}/N\mbb{Z}) \mid 
N\ge 3 \}. 
\end{align*}
Here, $\varphi$ is the Euler's totient function. 
The lists of $\Phi(n)$, $H(n)$ 
(and $G(n)$ with $p\not=3$ appeared in Remark \ref{MT:rem})
for small $n$ are given at the end of this paper.
In this section, we study some upper bounds of $\Phi$ and $H$.

\subsection{The function $H$}
For the function $H$,  we refer results of \cite[\S3 and \S4]{Si1}.
The exact formula for $H(n)$ is as follows:
\begin{align*}
H(n)=\frac{1}{2^{n-1}}\prod_{q} q^{r(q)}
\end{align*}
where the product is over primes $q\le 2n+1$, 
\begin{align*}
r(2) = \left[n \right]+\sum^{\infty}_{j=0} \left[ \frac{2n}{2^j}\right],\  
\mbox{and}\    
r(q) = \sum^{\infty}_{j=0} \left[ \frac{2n}{q^j(q-1)}\right] \ \mbox{if $q$ is odd}.
\end{align*}
Moreover,    we have
\begin{theorem}[{\cite[Corollary 3.3]{Si1}}]
\label{H:bound}
We have 
$$
H(n) < 2(9n)^{2n}
$$
 for any $n>0$.
\end{theorem}

\subsection{The function $\Phi$}
Next we consider the function $\Phi$. 
At first, we remark that $\Phi(n)$ must be even
since $\vphi(x)=\vphi(2x)$ if $x$ is odd. 
Furthermore, $\Phi(n)$ is not a power of $2$.
(In fact, we have $\vphi(2^r)=\vphi(2^{r-1}\cdot 3)$ if $r\ge 2$.)
Thus it holds that 
\begin{equation}
 \Phi(n) = \mrm{Max}\left\{m\in \mbb{Z}  \middle|
  \begin{aligned} 
    \mbox{$\varphi(m)$ divides $2n$, and $m=2^rx$} \\
    \mbox{where  $r\ge 1$ and   $x\ge 3$ is odd}
  \end{aligned}
  \right\}.
\end{equation}
We show some elementary formulas.
\begin{proposition}
\label{P:bound1}
{\rm (1)}  We have $\Phi(1)=6$ and  $6\le \Phi(n)<6n\sqrt[3]{n}$ for $n>1$.

\noindent
{\rm (2)} Put $t=v_2(n)+2$ and let $p_1=2 < p_2< \cdots <p_t$ be the first $t$ prime numbers. 
Then we have 
$$
\Phi(n)\le 2n\prod^t_{i=1}\frac{p_i}{p_i-1}.
$$
In particular, we have $\Phi(n)\le 6n$ if $n$ is odd.

\noindent
{\rm (3)} If $n>3$ is an odd prime, we have\footnote{
A prime number $p$ is called a {\it Sophie German prime} 
if $2p+1$ is also prime. 
It is not known whether there exist infinitely many Sophie German prime or not.
On the other hand, there exist infinitely many prime which is not Sophie German prime.
In fact, every prime number $p$ with $p\equiv 1 \ \mrm{mod}\ 3$ 
is not Sophie German prime.}
\begin{align*}
\Phi(n)
=
\left\{
\begin{array}{cl}
6 &\quad 
\mbox{if $2n+1$ is not prime}, \cr
4n+2 
&\quad 
\mbox{if $2n+1$ is prime}.
\end{array}
\right.
\end{align*}

\end{proposition}

\begin{proof}
To check $\Phi(1)=6$ is an easy exercise.
Since $\vphi(6)=2\mid 2n$,
we have $\Phi(n)\ge 6$ for any $n$. 
Suppose $n>1$. 
We take an even integer $m>0$, which  is  of the form $2^rx$  
where  $r\ge 1$ and  $x\ge 3$ is odd, 
such that $\vphi(m)\mid 2n$. 
Let $m=2^r \prod^s_{i=1}q_i^{e_i}$ 
be the prime factorization of $m$ with $r,s,e_1, \dots ,e_s\ge 1$.
Since $\vphi(m)=2^{r-1} \prod^s_{i=1}q_i^{e_i-1}(q_i-1)$ and $\vphi(m)\mid 2n$,
we have $v_2(2n)\ge r-1+s$ and thus 
\begin{equation}
\label{pf}
r+s\le v_2(n)+2.
\end{equation}
Then we find 
$$
2n\ge \vphi(m)=m \left(1-\frac{1}{2}\right)\prod^s_{i=1}\left(1-\frac{1}{q_i}\right)
\ge m \prod^{s+1}_{i=1}\left(1-\frac{1}{p_i}\right)
\ge m \prod^{t}_{i=1}\left(1-\frac{1}{p_i}\right).
$$
This shows (2).  Furthermore, we have
\begin{align*}
\Phi(n) & \le 2n\prod^t_{i=1}\frac{p_i}{p_i-1}
=6n\prod^t_{i=3}\frac{p_i}{p_i-1}
\le 6n \left(\frac{5}{5-1}\right)^{v_2(n)} \\
& \le 6n \cdot   \left(\frac{5}{4}\right)^{\log_2(n)}
< 6n \cdot   2^{\frac{1}{3}\log_2(n)}.
\end{align*}
Thus we obtain (1).
Let us show (3).  From now on we assume that  $n>3$ is an odd prime.
Assume that $m\not=6$.
Since $n$ is odd, it follows from \eqref{pf} that the prime factorization of $m$ is 
of the form $m=2 q^e$ for some odd prime $q$. 
Then $\frac{1}{2}\vphi(m)=q^{e-1}\frac{q-1}{2}$ divides $n$.
Since $n>3$ is a prime and $m\not=6$, we find  $e=1$ and  $\frac{q-1}{2}=n$.
This implies  $2n+1$ must be prime and $m=4n+2$.
Now the result follows.
\end{proof}

Let us consider an upper bound  of $\Phi$ 
by using an "analytic" lower bound  function of $\vphi$ given by Rosser and Schoenfeld.  
If we denote by $\gamma$  the Euler's constant\footnote{
$\displaystyle \gamma=\int^{\infty}_{1}\left(\frac{1}{[x]}-\frac{1}{x} \right)  dx
=0.57721\cdots $. Note also 
$\displaystyle e^{\gamma}=1.78107\cdots$.}, 
it is shown in \cite[Theorem 15]{RS}  that 
we have\footnote{
More precisely, Theorem 15 of \cite{RS} states that 
\begin{equation}
\label{phi:lower}
\varphi(m)>\frac{m}{e^{\gamma}\log \log m+\frac{5}{2\log \log m}}
\end{equation}
for $m\ge 3$ 
except
when $m$ is 
the product of the first nine primes 
$m=223092870=2\cdot 3\cdot 5\cdot \cdot 7
\cdot 11\cdot 13\cdot 17\cdot 19\cdot 23$.}
\begin{equation*}
\varphi(m)>\frac{m}{e^{\gamma}\log \log m+\frac{3}{\log \log m}}
\end{equation*}
for $m\ge 3$.   
We set
$$
\Psi(n):=\mrm{Max}\{m\in \mbb{Z} \mid \mbox{$\varphi(m)\le 2n$} \}.
$$
We clearly have 
$\Phi(n)\le \Psi(n)$ for all $n>0$.
\begin{proposition}
\label{P:bound2}
For any real number $C>2e^{\gamma}$,
we have 
$$
\Psi(n)<Cn\log \log n
$$ 
for  any $n$ large enough.
\end{proposition}

\begin{proof}
Put $f(x)=C\log \log x$.
Take any integer $N>0$ which satisfies the following properties: 
For all $x>N$, it holds 
\begin{itemize}
\item[(i)] $f(x)>\frac{1}{x}e^{e^2}$ and 
\item[(ii)] $f(x)> 2e^{\gamma} (\log \log (xf(x))+1)$.
\end{itemize}
(The assumption $C>2e^{\gamma}$ asserts the existence of such $N$.)
Take any integer $n>N$. 
It suffices to show that $n$ satisfies the desired inequality. 
Assume that there exists an integer $m$ such that 
both $\vphi(m)\le 2n$ and $m\ge nf(n)$ hold.
Since we have $e^{\gamma} >\frac{3}{\log \log x}$ for $x>e^{e^2}$
and $m(\ge nf(n))>e^{e^2}$,
we find
\begin{equation*}
\frac{1}{e^{\gamma}} \cdot \frac{m}{\log \log m + 1}< 
\frac{m}{e^{\gamma}\log \log m+\frac{3}{\log \log m}}<\varphi(m)\le 2n.
\end{equation*}
We also have $\frac{nf(n)}{\log \log (nf(n))+1} 
\le \frac{m}{\log \log m+1}$
since the function $\frac{x}{\log \log x+1}$ is strictly increasing for $x>e$
and $m\ge nf(n) (>e^{e^2})>e$.
Hence we obtain 
\begin{equation*}
\frac{1}{e^{\gamma}} \cdot \frac{nf(n)}{\log \log (nf(n))+1}< 2n,
\end{equation*}
which gives $f(n)< 2e^{\gamma} (\log \log (nf(n))+1)$.
This contradicts the condition (ii).
Therefore, we conclude that, if $\vphi(m)\le 2n$,
then it holds $m< nf(n)$.
This implies $\Psi(n)<nf(n)=Cn\log \log n$.
\end{proof}

\begin{remark}
Consider the case $C =4$. 
By studying  (i) and (ii) in the above proof more carefully, 
we can show 
$$
\Psi(n)<4n\log \log n
$$ 
for  any $n>e^{(1.001e)^9}$. 

Here we check the above inequality. 
The condition (ii) is equivalent to say that 
$$ 
(\log x)^{\frac{C}{2e^{\gamma}}-1}> e \left(1+\frac{\log(C \log \log x)}{\log x}\right).
$$
We assume $x>e^{e^9}$.
Since $\frac{C}{2e^{\gamma}}-1>\frac{4}{3.6}-1=\frac{1}{9}$
and $\frac{\log(C \log \log x)}{\log x}<0.001$, the inequality (ii) holds
if $(\log x)^{\frac{1}{9}}> 1.001e$, that is, $x>e^{(1.001e)^9}$.
Note that (i) clearly holds for such $x$.
\end{remark}

{\footnotesize
\begin{table}[htb]
\centering
\caption{$\Phi(n)$}
\label{List1}
  \begin{tabular}{|ll|ll|ll|}  \hline
$n$	&		\multicolumn{1}{c|}{$\Phi(n)$}					& $n$	&		\multicolumn{1}{c|}{$\Phi(n)$}				&	$n$	&		\multicolumn{1}{c|}{$\Phi(n)$}					\\\hline\hline	
1	&	$	2^1\cdot 3^1	$	&	41	&	$	2^1\cdot 83^1	$	&	81	&	$	2^1\cdot 3^5	$	\\	\hline
2	&	$	2^2\cdot 3^1	$	&	42	&	$	2^1\cdot 3^1\cdot 7^2	$	&	82	&	$	2^1\cdot 3^1\cdot 83^1	$	\\	\hline
3	&	$	2^1\cdot 3^2	$	&	43	&	$	2^1\cdot 3^1	$	&	83	&	$	2^1\cdot 167^1	$	\\	\hline
4	&	$	2^1\cdot 3^1\cdot 5^1	$	&	44	&	$	2^2\cdot 3^1\cdot 23^1	$	&	84	&	$	2^2\cdot 3^1\cdot 7^2	$	\\	\hline
5	&	$	2^1\cdot 11^1	$	&	45	&	$	2^1\cdot 31^1	$	&	85	&	$	2^1\cdot 11^1	$	\\	\hline
6	&	$	2^1\cdot 3^1\cdot 7^1	$	&	46	&	$	2^1\cdot 3^1\cdot 47^1	$	&	86	&	$	2^1\cdot 173^1	$	\\	\hline
7	&	$	2^1\cdot 3^1	$	&	47	&	$	2^1\cdot 3^1	$	&	87	&	$	2^1\cdot 59^1	$	\\	\hline
8	&	$	2^2\cdot 3^1\cdot 5^1	$	&	48	&	$	2^2\cdot 3^1\cdot 5^1\cdot 7^1	$	&	88	&	$	2^1\cdot 3^1\cdot 5^1\cdot 23^1	$	\\	\hline
9	&	$	2^1\cdot 3^3	$	&	49	&	$	2^1\cdot 3^1	$	&	89	&	$	2^1\cdot 179^1	$	\\	\hline
10	&	$	2^1\cdot 3^1\cdot 11^1	$	&	50	&	$	2^1\cdot 5^3	$	&	90	&	$	2^1\cdot 3^3\cdot 11^1	$	\\	\hline
11	&	$	2^1\cdot 23^1	$	&	51	&	$	2^1\cdot 103^1	$	&	91	&	$	2^1\cdot 3^1	$	\\	\hline
12	&	$	2^1\cdot 3^2\cdot 5^1	$	&	52	&	$	2^1\cdot 3^1\cdot 53^1	$	&	92	&	$	2^2\cdot 3^1\cdot 47^1	$	\\	\hline
13	&	$	2^1\cdot 3^1	$	&	53	&	$	2^1\cdot 107^1	$	&	93	&	$	2^1\cdot 3^2	$	\\	\hline
14	&	$	2^1\cdot 29^1	$	&	54	&	$	2^1\cdot 3^3\cdot 7^1	$	&	94	&	$	2^2\cdot 3^1	$	\\	\hline
15	&	$	2^1\cdot 31^1	$	&	55	&	$	2^1\cdot 11^2	$	&	95	&	$	2^1\cdot 191^1	$	\\	\hline
16	&	$	2^3\cdot 3^1\cdot 5^1	$	&	56	&	$	2^2\cdot 3^1\cdot 29^1	$	&	96	&	$	2^3\cdot 3^1\cdot 5^1\cdot 7^1	$	\\	\hline
17	&	$	2^1\cdot 3^1	$	&	57	&	$	2^1\cdot 3^2	$	&	97	&	$	2^1\cdot 3^1	$	\\	\hline
18	&	$	2^1\cdot 3^2\cdot 7^1	$	&	58	&	$	2^1\cdot 3^1\cdot 59^1	$	&	98	&	$	2^1\cdot 197^1	$	\\	\hline
19	&	$	2^1\cdot 3^1	$	&	59	&	$	2^1\cdot 3^1	$	&	99	&	$	2^1\cdot 199^1	$	\\	\hline
20	&	$	2^1\cdot 3^1\cdot 5^2	$	&	60	&	$	2^1\cdot 3^1\cdot 7^1\cdot 11^1	$	&	100	&	$	2^1\cdot 3^1\cdot 5^3	$	\\	\hline
21	&	$	2^1\cdot 7^2	$	&	61	&	$	2^1\cdot 3^1	$	&	101	&	$	2^1\cdot 3^1	$	\\	\hline
22	&	$	2^1\cdot 3^1\cdot 23^1	$	&	62	&	$	2^2\cdot 3^1	$	&	102	&	$	2^1\cdot 3^1\cdot 103^1	$	\\	\hline
23	&	$	2^1\cdot 47^1	$	&	63	&	$	2^1\cdot 127^1	$	&	103	&	$	2^1\cdot 3^1	$	\\	\hline
24	&	$	2^1\cdot 3^1\cdot 5^1\cdot 7^1	$	&	64	&	$	2^1\cdot 3^1\cdot 5^1\cdot 17^1	$	&	104	&	$	2^2\cdot 3^1\cdot 53^1	$	\\	\hline
25	&	$	2^1\cdot 11^1	$	&	65	&	$	2^1\cdot 131^1	$	&	105	&	$	2^1\cdot 211^1	$	\\	\hline
26	&	$	2^1\cdot 53^1	$	&	66	&	$	2^1\cdot 3^2\cdot 23^1	$	&	106	&	$	2^1\cdot 3^1\cdot 107^1	$	\\	\hline
27	&	$	2^1\cdot 3^4	$	&	67	&	$	2^1\cdot 3^1	$	&	107	&	$	2^1\cdot 3^1	$	\\	\hline
28	&	$	2^1\cdot 3^1\cdot 29^1	$	&	68	&	$	2^1\cdot 137^1	$	&	108	&	$	2^1\cdot 3^4\cdot 5^1	$	\\	\hline
29	&	$	2^1\cdot 59^1	$	&	69	&	$	2^1\cdot 139^1	$	&	109	&	$	2^1\cdot 3^1	$	\\	\hline
30	&	$	2^1\cdot 3^2\cdot 11^1	$	&	70	&	$	2^1\cdot 3^1\cdot 71^1	$	&	110	&	$	2^1\cdot 3^1\cdot 11^2	$	\\	\hline
31	&	$	2^1\cdot 3^1	$	&	71	&	$	2^1\cdot 3^1	$	&	111	&	$	2^1\cdot 223^1	$	\\	\hline
32	&	$	2^4\cdot 3^1\cdot 5^1	$	&	72	&	$	2^1\cdot 3^2\cdot 5^1\cdot 7^1	$	&	112	&	$	2^1\cdot 3^1\cdot 5^1\cdot 29^1	$	\\	\hline
33	&	$	2^1\cdot 67^1	$	&	73	&	$	2^1\cdot 3^1	$	&	113	&	$	2^1\cdot 227^1	$	\\	\hline
34	&	$	2^2\cdot 3^1	$	&	74	&	$	2^1\cdot 149^1	$	&	114	&	$	2^1\cdot 229^1	$	\\	\hline
35	&	$	2^1\cdot 71^1	$	&	75	&	$	2^1\cdot 151^1	$	&	115	&	$	2^1\cdot 47^1	$	\\	\hline
36	&	$	2^1\cdot 3^3\cdot 5^1	$	&	76	&	$	2^1\cdot 3^1\cdot 5^1	$	&	116	&	$	2^2\cdot 3^1\cdot 59^1	$	\\	\hline
37	&	$	2^1\cdot 3^1	$	&	77	&	$	2^1\cdot 23^1	$	&	117	&	$	2^1\cdot 79^1	$	\\	\hline
38	&	$	2^2\cdot 3^1	$	&	78	&	$	2^1\cdot 3^1\cdot 79^1	$	&	118	&	$	2^2\cdot 3^1	$	\\	\hline
39	&	$	2^1\cdot 79^1	$	&	79	&	$	2^1\cdot 3^1	$	&	119	&	$	2^1\cdot 239^1	$	\\	\hline
40	&	$	2^1\cdot 3^1\cdot 5^1\cdot 11^1	$	&	80	&	$	2^2\cdot 3^1\cdot 5^1\cdot 11^1	$	&	120	&	$	2^1\cdot 3^1\cdot 5^2\cdot 7^1	$	\\	\hline
		\end{tabular} 
\end{table}
}
\clearpage

{\footnotesize
\begin{table}[htb]
\centering
\caption{$H(n)$}
\label{List2}
  \begin{tabular}{|l l|}  \hline
$n$	&	\multicolumn{1}{c|}{$H(n)$}			\\\hline\hline	
1	&	$	2^4\cdot 3^1	$	\\	\hline
2	&	$	2^8\cdot 3^2\cdot 5^1	$	\\	\hline
3	&	$	2^{11}\cdot 3^4\cdot 5^1\cdot 7^1	$	\\	\hline
4	&	$	2^{16}\cdot 3^5\cdot 5^2\cdot 7^1	$	\\	\hline
5	&	$	2^{19}\cdot 3^6\cdot 5^2\cdot 7^1\cdot 11^1	$	\\	\hline
6	&	$	2^{23}\cdot 3^8\cdot 5^3\cdot 7^2\cdot 11^1\cdot 13^1	$	\\	\hline
7	&	$	2^{26}\cdot 3^9\cdot 5^3\cdot 7^2\cdot 11^1\cdot 13^1	$	\\	\hline
8	&	$	2^{32}\cdot 3^{10}\cdot 5^4\cdot 7^2\cdot 11^1\cdot 13^1\cdot 17^1	$	\\	\hline
9	&	$	2^{35}\cdot 3^{13}\cdot 5^4\cdot 7^3\cdot 11^1\cdot 13^1\cdot 17^1\cdot 19^1	$	\\	\hline
10	&	$	2^{39}\cdot 3^{14}\cdot 5^6\cdot 7^3\cdot 11^2\cdot 13^1\cdot 17^1\cdot 19^1	$	\\	\hline
11	&	$	2^{42}\cdot 3^{15}\cdot 5^6\cdot 7^3\cdot 11^2\cdot 13^1\cdot 17^1\cdot 19^1\cdot 23^1	$	\\	\hline
12	&	$	2^{47}\cdot 3^{17}\cdot 5^7\cdot 7^4\cdot 11^2\cdot 13^2\cdot 17^1\cdot 19^1\cdot 23^1	$	\\	\hline
13	&	$	2^{50}\cdot 3^{18}\cdot 5^7\cdot 7^4\cdot 11^2\cdot 13^2\cdot 17^1\cdot 19^1\cdot 23^1	$	\\	\hline
14	&	$	2^{54}\cdot 3^{19}\cdot 5^8\cdot 7^4\cdot 11^2\cdot 13^2\cdot 17^1\cdot 19^1\cdot 23^1\cdot 29^1	$	\\	\hline
15	&	$	2^{57}\cdot 3^{21}\cdot 5^8\cdot 7^5\cdot 11^3\cdot 13^2\cdot 17^1\cdot 19^1\cdot 23^1\cdot 29^1\cdot 31^1	$	\\	\hline
16	&	$	2^{64}\cdot 3^{22}\cdot 5^9\cdot 7^5\cdot 11^3\cdot 13^2\cdot 17^2\cdot 19^1\cdot 23^1\cdot 29^1\cdot 31^1	$	\\	\hline
17	&	$	2^{67}\cdot 3^{23}\cdot 5^9\cdot 7^5\cdot 11^3\cdot 13^2\cdot 17^2\cdot 19^1\cdot 23^1\cdot 29^1\cdot 31^1	$	\\	\hline
18	&	$	2^{71}\cdot 3^{26}\cdot 5^{10}\cdot 7^6\cdot 11^3\cdot 13^3\cdot 17^2\cdot 19^2\cdot 23^1\cdot 29^1\cdot 31^1\cdot 37^1	$	\\	\hline
19	&	$	2^{74}\cdot 3^{27}\cdot 5^{10}\cdot 7^6\cdot 11^3\cdot 13^3\cdot 17^2\cdot 19^2\cdot 23^1\cdot 29^1\cdot 31^1\cdot 37^1	$	\\	\hline
20	&	$	2^{79}\cdot 3^{28}\cdot 5^{12}\cdot 7^6\cdot 11^4\cdot 13^3\cdot 17^2\cdot 19^2\cdot 23^1\cdot 29^1\cdot 31^1\cdot 37^1\cdot 41^1	$	\\	\hline
21	&	$	2^{82}\cdot 3^{30}\cdot 5^{12}\cdot 7^8\cdot 11^4\cdot 13^3\cdot 17^2\cdot 19^2\cdot 23^1\cdot 29^1\cdot 31^1\cdot 37^1\cdot 41^1\cdot 43^1	$	\\	\hline
22	&	$	2^{86}\cdot 3^{31}\cdot 5^{13}\cdot 7^8\cdot 11^4\cdot 13^3\cdot 17^2\cdot 19^2\cdot 23^2\cdot 29^1\cdot 31^1\cdot 37^1\cdot 41^1\cdot 43^1	$	\\	\hline
23	&	$	2^{89}\cdot 3^{32}\cdot 5^{13}\cdot 7^8\cdot 11^4\cdot 13^3\cdot 17^2\cdot 19^2\cdot 23^2\cdot 29^1\cdot 31^1\cdot 37^1\cdot 41^1\cdot 43^1\cdot 47^1	$	\\	\hline
24	&	$	2^{95}\cdot 3^{34}\cdot 5^{14}\cdot 7^9\cdot 11^4\cdot 13^4\cdot 17^3\cdot 19^2\cdot 23^2\cdot 29^1\cdot 31^1\cdot 37^1\cdot 41^1\cdot 43^1\cdot 47^1	$	\\	\hline
25	&	$	2^{98}\cdot 3^{35}\cdot 5^{14}\cdot 7^9\cdot 11^5\cdot 13^4\cdot 17^3\cdot 19^2\cdot 23^2\cdot 29^1\cdot 31^1\cdot 37^1\cdot 41^1\cdot 43^1\cdot 47^1	$	\\	\hline
		\end{tabular} 
\end{table}
}

{\footnotesize
\begin{table}[htb]
\centering
\caption{$G(n)$ (for $p\not= 3$)} 
\label{List3}
  \begin{tabular}{|l l|}  \hline
$n$	&	\multicolumn{1}{c|}{$G(n)$}			\\\hline\hline	
1	&	$	2^4\cdot 3^1	$	\\	\hline
2	&	$	2^9\cdot 3^6\cdot 5^1\cdot 13^1	$	\\	\hline
3	&	$	2^{13}\cdot 3^{15}\cdot 5^1\cdot 7^1\cdot 11^2\cdot 13^2	$	\\	\hline
4	&	$	2^{19}\cdot 3^{28}\cdot 5^2\cdot 7^1\cdot 11^2\cdot 13^2\cdot 41^1\cdot 1093^1	$	\\	\hline
5	&	$	2^{23}\cdot 3^{45}\cdot 5^2\cdot 7^1\cdot 11^4\cdot 13^3\cdot 41^1\cdot 61^1\cdot 757^1\cdot 1093^1	$	\\	\hline
6	&	$	2^{28}\cdot 3^{66}\cdot 5^3\cdot 7^2\cdot 11^4\cdot 13^4\cdot 23^1\cdot 41^1\cdot 61^1\cdot 73^1\cdot 757^1\cdot 1093^1\cdot 3851^1	$	\\	\hline
7	&	$	2^{32}\cdot 3^{91}\cdot 5^3\cdot 7^2\cdot 11^4\cdot 13^4\cdot 23^1\cdot 41^1\cdot 61^1\cdot 73^1\cdot 547^1\cdot 757^1\cdot 1093^2\cdot 3851^1\cdot 797161^1	$	\\	\hline
		\end{tabular} 
\end{table}
}


\clearpage

\begin{thebibliography}{1000}

\bibitem[CCM]{CCM}
M,\ Chou, P.-L. Clark and M.\ Milosevic,
\emph{Acyclotomy of torsion in the CM case}, 
Ramanujan J. 55, {\bf 3} (2021), 1015--1037.

\if0
\bibitem[CD]{CD}
I.-D.\ Corso and R. Dvornicich, 
\emph{The compositum of wild extensions of local fields of prime
              degree}, 
Monatsh. Math. 150, {\bf 4} (2007), 271--288.
\fi

\bibitem[Ch]{Ch}
M.\ Chou, 
\emph{Torsion of rational elliptic curves over the maximal abelian extension of $\mbb{Q}$}, 
Pacific J. Math. 302, no. 2 (2019), 481--509.

\if0
\bibitem[CLS]{CLS}
B. Chiarellotto and B. Le Stum, 
\emph{Sur la puret\'e de la cohomologie cristalline}, 
C. R. Acad. Sci. Paris S\'er. I Math. {\bf 8} (1998), 
961--963.
\fi

\bibitem[Co]{Co}
B.\ Conrad,
\emph{Lifting global representations with local properties}, 
preprint, 2011, available at 
http://math.stanford.edu/\textasciitilde conrad/papers/locchar.pdf



\bibitem[CX]{CX}
P.-L.\ Clark and X.\ Xarles, 
\emph{Local bounds for torsion points on abelian varieties},
Canad. J. Math. vol. {\bf 60} (2008), 
532--555.

\if0
\bibitem[Fa1]{Fa1}
G. Faltings,
\emph{{$p$}-adic {H}odge theory},
J. Amer. Math. Soc. {\bf 1} (1988),
255--299. 
\fi

\if0
\bibitem[Fa2]{Fa}
G. Faltings,
\emph{Crystalline cohomology and $p$-adic Galois-representations}, 
Algebraic analysis, geometry, and number theory ({B}altimore {MD}, 1988),
25--80.
\fi


\bibitem[Gr]{Gr}
Alexander Grothendieck, 
\emph{Mod\`eles de N\'eron et monodromie, in Groupes de monodromie
en g\'eometrie alg\'ebrique, SGA 7}, 
Lecture Notes in Mathematics,\ vol.\ {\bf 288},
313--523 (1972).


\bibitem[Im]{Im}
H.\ Imai, 
\emph{A remark on the rational points of abelian varieties with values in 
cyclotomic $\mbb{Z}_p$-extensions},
Proc.\ Japan Acad.\ {\bf 51} (1975), 12--16. 


\if0
\bibitem[JR]{JR}
J.-W.\  Jones and D.-P.\ Roberts, 
\emph{A database of local fields},
J.\ Symbolic Comput.\   {\bf 1} (2006), 80--97. 
\fi


\bibitem[KL]{KL}
N.-M.\ Katz and S.\ Lang, 
\emph{Finiteness theorems in geometric classfield theory}, 
with an appendix by Kenneth A.\ Ribet, Enseign. Math. (2) 27 (1981), 
no. 3-4, 285--319.

\bibitem[KT]{KT}
Y. Kubo and Y. Taguchi,
\emph{A generalization of a theorem of Imai and its applications to Iwasawa theory},
Math.\ Z.\ {\bf 275} (2013), 1181--1195.

\if0
\bibitem[KM]{KM}
N. Katz and W. Messing,
\emph{Some consequences of the {R}iemann hypothesis for varieties over finite fields},
Invent. Math. {\bf 23} (1974), 73--77.
\fi


\bibitem[Mat]{Mat}
A.\ Mattuck,
\emph{Abelian varieties over $p$-adic ground fields},
Ann. of Math. (2) {\bf 62} (1955), 92--119.

\if0
\bibitem[Maz]{Maz}
B.\ Mazur, 
\emph{Rational isogenies of prime degree},
with an appendix by D. Goldfeld,
Invent. Math. 44 (1978), no. 2, 129--162.
\fi 

\if0
\bibitem[Na]{Na}
Y. Nakkajima,
\emph{$p$-adic weight spectral sequences of log varieties},
J. Math. Sci. Univ. Tokyo {\bf 12} (2005),
513--661.
\fi


\bibitem[Ne]{Ne}
J.\ Neukirch,
\emph{Algebraic Number Theory},
Translated from the 1992 German original and with a note by Norbert Schappacher. With a foreword by G. Harder. Grundlehren der Mathematischen Wissenschaften [Fundamental Principles of Mathematical Sciences], 322. Springer-Verlag, Berlin, 1999.


\if0
\bibitem[Ni]{Ni}
W. Niziol,
\emph{Crystalline conjecture via {$K$}-theory},
Ann. Sci. \'{E}cole Norm. Sup. (4) {\bf 31} (1998),
659--681.
\fi



\bibitem[Oz1]{Oz}
Y.\ Ozeki,
\emph{Torsion of abelian varieties and Lubin-Tate extensions}, 
J. Number Theory {\bf 207} (2020), 
282--293.

\bibitem[Oz2]{Oz1}
Y.\ Ozeki,
\emph{Bounds on torsion of CM abelian varieties over a $p$-adic field
with values in a field of $p$-power roots},
arXiv:2209.01811

\bibitem[RS]{RS}
J.-B.\ Rosser and L.\ Schoenfeld,
\emph{Approximate formulas for some functions of prime numbers},
Illinois J. Math. {\bf 6} (1962), 64--94. 


\bibitem[Se]{Se1}
J.-P.\ Serre,
\emph{Local fields},
Translated from the French by Marvin Jay Greenberg. 
Graduate Texts in Mathematics, 67. 
Springer-Verlag, New York-Berlin, 1979.

\if0
\bibitem[Se2]{Se2}
J.-P.\ Serre,
\emph{Abelian $l$-adic representations and elliptic curves},
second ed., 
Advanced Book Classics, 
Addison-Wesley Publishing Company Advanced Book Program,
Redwood City, CA, 1989,
With the collaboration of Willem Kuyk and John Labute.
\fi

\bibitem[Si1]{Si1}
A.\ Silverberg,
\emph{Fields of definition for homomorphisms of abelian varieties},
J. Pure Appl. Algebra 77 (1992), no. 3, 253--262.


\bibitem[Si2]{Si2}
A.\ Silverberg,
\emph{Points of finite order on abelian varieties},
Contemp. Math. {\bf 133} (1992), 
175--193.


\bibitem[Si3]{Si3}
A.\ Silverberg,
\emph{Open questions in arithmetic algebraic geometry},
In: Arithmetic Algebraic Geometry.
IAS/Park CityMath. Ser. 9, American Mathematical Society, Providence, RI, 2001, 
83--142.

\bibitem[ST]{ST}
J.-P.\ Serre and J.\ Tate,
\emph{Good reduction of abelian varieties},
Ann. of Math. (2) {\bf 8} (1986),
492--517.


\if0
\bibitem[Wa]{Wa}
William C.\ Waterhouse, 
\emph{Abelian varieties over finite fields},
Ann. Sci. \'{E}cole Norm. Sup. (4) {\bf 2} (1969),
521--560.
\fi

\end{thebibliography}
\end{document}